\documentclass[12pt]{article}
\usepackage[english]{babel}
\usepackage{amsmath,amsfonts,amssymb,amstext,amsmath,amsthm,amscd}

\usepackage{mathtools}
\usepackage{mathrsfs}
\usepackage{dsfont}

\usepackage{indentfirst}
\usepackage{latexsym}
\usepackage{sidecap}
\usepackage{booktabs}

\usepackage{setspace}
\usepackage{calrsfs}
\usepackage{hyperref}
\hypersetup{
    colorlinks=true,
    linkcolor=red,
    filecolor=magenta,      
    urlcolor=cyan,
}
\usepackage{mathrsfs}
\usepackage{textcomp}
\usepackage{braket}
\usepackage{colortbl}
\usepackage{bbm}

\theoremstyle{plain}
\newtheorem{Theorem}{Theorem}[section]
\theoremstyle{definition}

\theoremstyle{remark}
\newtheorem{Remark}[Theorem]{Remark}
\theoremstyle{remark}

\theoremstyle{plain}

\theoremstyle{plain}

\theoremstyle{plain}

\theoremstyle{remark}
\newtheorem{Example}{Example}[section]
\theoremstyle{remark}

\theoremstyle{remark}

\theoremstyle{remark}
\newtheorem{Assumptions}[Theorem]{Assumptions}

\newcommand\RR{\mathbb{R}}

\DeclareMathAlphabet{\pazocal}{OMS}{zplm}{m}{n}

\usepackage{tikz}
\usetikzlibrary{arrows}

\newcommand{\Ind}[1]{\mathbbm{1}_{\left \{#1\right \}}}

\newcommand{\BD}{\mathbf{B}}
\newcommand{\BV}{\mathbf{\Sigma}}
\newcommand{\BX}{\mathbf{X}}
\newcommand{\BL}{\mathbf{L}}
\newcommand{\BG}{\mathbf{G}}
\newcommand{\BW}{\mathbf{W}}

\usepackage{fancyhdr}
\usepackage{calc} 
\begin{document}

\title{A maximum principle for a stochastic control problem with multiple random terminal times}
\author{Francesco Cordoni$^{a}$ \and Luca Di Persio$^{a}$}
\date{}
\maketitle

\renewcommand{\thefootnote}{\fnsymbol{footnote}}
\footnotetext{{\scriptsize $^{a}$ Department of Computer Science, University of Verona, Strada le Grazie, 15, Verona, 37134, Italy}}
\footnotetext{{\scriptsize E-mail addresses: francescogiuseppe.cordoni@univr.it
(Francesco Cordoni), luca.dipersio@univr.it (Luca Di Persio)}}

\begin{abstract}
In the present paper we derive, via a backward induction technique, an ad hoc maximum principle for an optimal control problem with multiple random terminal times. We thus apply the aforementioned result to the case of a linear quadratic controller, providing solutions for the optimal control in terms of Riccati backward SDE with random terminal time. 
\end{abstract}

\textbf{AMS Classification subjects:} 93E03, 93E20, 60H10, . \medskip

\textbf{Keywords or phrases: }Stochastic optimal control, multiple defaults time, maximum principle, linear--quadratic controller, financial network.

\section{Introduction}

In the last decades stochastic optimal control theory has received an increasing attention by the mathematical community, also in connection with several concrete applications, spanning from industry to finance, from biology to crowd dyamics, etc. In all of above applications a rigorous theory of stochastic optimal control (SOC), under suitable assumption on the {\it source of random noise}, revealed to be a fundamental point. 

To this aim different theoretical approaches have been developed. They can be broadly divided into two classes: partial differential equations (PDE) methods via the \textit{Hamilton-Jacobi-Bellman} (HJB) equation, and methods based on the \textit{maximum principle} via backward stochastic differential equations (BSDEs), see, e.g.,  \cite{Fle,PhaBook,Yon} 
In particular BSDEs' methods have  proved to be particularly adapted for a large set of SOC-problems, as reported, e.g., in
\cite{PhaR}.
Within previously mentioned problems a particular role is played by those SOC problems characterized by
the specification of a random terminal time. In particular, 
this a classical task in Finance at least since the  
recent financial credit crunch which imposed the need to model possible defaults and credit risks.
When dealing with optimal control with random terminal time, two main approaches are possibile. The first possible setting considers the random terminal time as a to be completely inaccessible to the reference filtration. The related classical approach 
consists in enlarging the reference filtration, see, e.g. \cite{Man}. In this way, via a suitable \textit{density assumption} on the conditional law of the random time,  the original problem is converted into a control problem with fixed terminal time, with respect to the new enlarged filtration, see, e.g., \cite{ElK,Pha} for more theoretical insights and to \cite{BJR,BR,CDP,JKP} for some concrete applications. 

A second, alternative,  approach assumes 
that the stopping times are accessible from the reference filtration, hence implying  
a {\it perfect information} about the triggered random times. The typical assumption in this setting is that the stopping time $\tau$ is defined as the first hitting time of a barrier $v$ for a reference system whose dynamic is given by a  \textit{stochastic differential equation} (SDE). 
In a credit risk setting, such an approach is known as the \textit{structural approach},  and it has a long-standing financial literature whose first results date back to \cite{Mer}. It is worth stressing that this last scenario does not fall back into previous one, where inaccessible stopping times are considered. In fact, if the stopping time is to be defined as the first hitting time, it does not satisfy above mention \textit{density hypothesis}.

The present paper investigates a SOC-problem with multiple random events of the latter type. Therefore, differently from \cite{JKP,Pha}, we will not assume random events to be totally inaccessible, but, instead, they will be defined as first hitting time, against a predetermined boundary, of the driving process.

In particular, we will consider a controlled system of $n \in \mathbb{N}$ SDEs of the general form
\begin{equation}\label{EQN:Introe}
\begin{cases}
dX^i (t) = \mu^i(t,X^i(t),\alpha^i(t))dt + \sigma^i(t,X^i(t),\alpha^i(t))dW^i(t)\, ,\quad i=1,\dots,n\, ,\\
X^i(0) = x^i\, ,
\end{cases}
\end{equation}
under standard assumptions of Lipschitz coefficients $\mu^i$ and $\sigma^i$ with at most linear growth, being $\alpha^i$ the control. The notation will be specified in detail within subsequent sections. 

We aim at minimizing the following functional up to a given stopping time $\tau$,
\[
J(x,\alpha) = \mathbb{E} \int_0^\tau L(t,X(t),\alpha(t)) dt + G(\tau,X(\tau))\, ,
\]
for some suitable functions $L$ and $G$, where we have denoted by $X(t) = \left (X^1(t),\dots,X^n(t)\right )$ and $\alpha(t) = \left (\alpha^1(t),\dots,\alpha^n(t)\right )$.

Then we assume that the system, instead of being stopped as soon as the stopping time $\tau$ is triggered, continues to evolve according to a new system of SDEs written as follows
\[
\begin{cases}
dX^i_1 (t) = \mu^i_1(t,X^i_1(t),\alpha^i_1(t))dt + \sigma^i_1(t,X^i_1(t),\alpha^i_1(t))dW^i(t)\, ,\quad i=1,\dots,n-1\, ,\\
X^i_1(\tau) = x^i_1\, ,
\end{cases}
\]
for some new coefficients $\mu_1^i$ and $\sigma_1^i$  again satisfying standard assumptions of linear growth and Lipschitz continuity. In particular, we will assume that, according to the triggered stopping time the $k-$th component in equation \eqref{EQN:Introe} has been set to $0$, according to rigorous definitions later specified. 
Then, we again aim at minimizing a functional of the form
\[
J_1(x_1,\alpha) = \mathbb{E} \int_{\tau}^{\tau_1} L_1(t,X_1(t),\alpha_1(t)) dt + G_1(\tau_1,X_1(\tau_1))\, ,
\]
with the same notation used before, $\tau_1$ being a new stopping time. 
We repeat such a scheme 
for a series of $n$ stopping times. Moreover, in complete generality, we assume that the order of the random times is not know a priori, hence forcing us to consider all possible combinations of random events with associated all the possible combinations of driving SDEs.

The main result of the present paper consists in deriving a \textit{stochastic maximum principle}, both in necessary and sufficient form, for the whole series of control problems stated above. 

Clearly, we cannot expect that the global optimal solution is given by gluing each optimal control between two consecutive stopping times. Instead, we will tackle the problem following a dynamic programming principle approach, as exploited, e.g., in \cite{Pha}. In particular,  we will solve the problem backward. Therefore, the case of all stopping times but one have been triggered is considered first, then we consider the problem with two random events left, etc.,  until the very first control problem. Following this scheme, we are able to provide the global optimal solution recursively, so that the $k-$th optimal control problem depends on the $(k+1)-$th optimal solution. We remark that altough the backward approach has been used in literature, see, e.g. \cite{JKP,Pha}, to the best of our knowledge the present work is the first one using such techniques where stopping times are defined as hitting times.

After having derived the main result, i.e. the aforementioned \textit{maximum principle}, we will consider the particular case of a linear--quadratic control problem, that is we assume the underlying dynamics to be linear in both the state variable and the control, with quadratic costs to be minimized. Such type of problems have been widely studied both from a theoretical and practical point of view since they often allow to obtain closed form solution for the optimal control. 

In particular, usually one can  write the solution to a linear--quadratic control problem in terms of the solution of a Riccati backward ordinary differential equation (ODE), hence reducing the original linear--quadratic stochastic control problem  to the solution of a simpler ODE, see, e.g., \cite{Yon} and \cite[Section 6.6]{PhaBook}, for possible financial applications. 
Let us recall that, considering either random coefficients for the driving equation or random terminal time in the control problem, the latter case being the one here treated, the backward Riccati ODE becomes a Riccati BSDE, see, e.g., \cite{GT1,GT2,KZ,KT}.

We stress that the results derived in the present paper find natural applications in many areas related to mathematical finance, and mainly related to systemic risk, where after recent credit crisis, the assumption of possibile failures has become the main ingredient in many robust financial models. Also, network models have seen an increasing mathematical attention during last years, as witnessed by the developmend of  several ad hoc techniques derived  to consider general dynamics on networks.  
We refer the interested reader to \cite{CDPN,CDPN1,DPZ},for general results on network models, and to \cite{Hur} for a financially oriented treatment. 

In particular, these models have proved to be particularly suitable if one is to consider a system of interconnected banks. Following thus the approach of \cite{Cap,Eis,Lip}, results derived in the present work can be successfully applied to a system of $n$ interconnected banks, lending and borrowing money.
As in \cite{Cap,CDPP} one can assume the presence of an external controller, typically called \textit{lender of last resort} (LOLR), who actively supervises the banks' system and possibly lending money to actors in needs. 
A standard assumption is that the LOLR lends money in order to optimize a given quadratic functional. Therefore, modelling the system as in \cite{CDPP}, we  recover a linear--quadratic setting allowing us to apply results obtained in the present work.\\

The  paper is organized as follows: in Section \ref{SEC:GS} we introduce the general setting, clarifying main assumptions; then, Section \ref{SEC:NMP} is devoted to the proof of the \textit{necessary maximum principle}, whereas in Section \ref{SEC:SMP} we will prove the \textit{sufficient maxim principle}; at last, in Section \ref{SEC:LQP1}, we apply previous results to the case of a linear--quadratic control problems also deriving the global solution by an interative scheme to solve 
a system of Riccati BSDEs.

\section{The general setting}\label{SEC:GS}

Let $n \in \mathbb{N}$ and $T<\infty$ a fixed terminal time and let us consider a standard complete filtered probability space $\left (\Omega,\mathcal{F},\left (\mathcal{F}_t\right )_{t \in [0,T]},\mathbb{P}\right )$ satisfying usual assumptions. 

In what follows we are going to consider a controlled system of $n$ SDEs, for $t \in [0,T]$ and  $ i = 1,\dots,n$, evolvong has folllows
\begin{equation}\label{EQN:SysiIniz1}
\begin{cases}
dX^{i;0} (t) &= \mu^{i;0}\left (t,X^{i;0}(t),\alpha^{i;0}(t) \right ) dt + \sigma^{i;0}\left (t,X^{i;0}(t),\alpha^{i;0}(t)\right )dW^{i}(t)\, ,\\
X^{i;0}(0)&=x^{i;0}_0\, ,
\end{cases}
\end{equation}
where $W^{i}(t)$ is a standard Brownian motion, $\alpha^{i;0}$ being the control. In particular, we assume
\[
\mathcal{A}^{i} := \left \{\alpha^{i;0} \in L^2_{ad}\left ([0,T];\RR\right ) \, : \, \alpha^{i;0}(t) \in A^{i} \,, \mbox{a.e.} \, t \in [0,T]\right \}\, ,
\]
where $A^{i} \subset \RR$ is assumed to be convex and closed, and we have denoted by $L^2_{ad}\left ([0,T];\RR \right )$ the space of $\left (\mathcal{F}_t\right )_{t \in [0,T]}$--adapted processes $\alpha$ such that
\[
\mathbb{E} \int_0^T |\alpha^{i;0}(t)|^2 dt<\infty\, ,
\]
while $A:= \otimes_{i=1}^n A^{i}$.

In what follows we will assume the following assumptions to hold.

\begin{Assumptions}\label{ASS:1}
Let $\mu:[0,T] \times \RR \times A \to \RR$ and $\sigma:[0,T]\times \RR \times A \to \RR$ be measurable functions and suppose that there exits a constant $C>0$ such that, for any $x$, $y \in \RR$, for any $a \in A$ and for any $t \in [0,T]$, it holds
\[
\begin{split}
&|\mu(t,x,a) -\mu(t,y,a)| + |\sigma(t,x,a) - \sigma(t,y,a)| \leq C |x-y|\, ,\\
&|\mu(t,x,a)| + |\sigma(t,x,a)| \leq C (1+|x|+ |a|)\, .
\end{split}
\]
\end{Assumptions}

We thus assume the coefficients $\mu^{i;0}$ and $\sigma^{i;0}$, for $i=1,\dots,n$, in equation \eqref{EQN:SysiIniz1}, satisfy assumptions \ref{ASS:1}. Thererfore, we have that there exists a unique strong solution to equation \eqref{EQN:SysiIniz1}, see, e.g., \cite{Fle,PhaBook}.

\begin{Remark}
In equation \eqref{EQN:SysiIniz1} we have considered an $\RR-$valued SDE, neverthelesse  what follows still holds if we consider a system of SDEs, each of which takes values in $\RR^{m_i}$, $m_i \in \mathbb{N}$, $i=1,\dots,n$.
\end{Remark}

Let us denote by 
\[
\begin{split}
\mathbf{X}^0(t) &= \left (X^{1;0}(t),\dots, X^{n;0}(t)\right )\, ,\\
\alpha^0(t) &=\left (\alpha^{1;0}(t),\dots,\alpha^{n;0}(t)\right )\,,
\end{split}
\]
then define the coefficients 
\[
\BD^0: [0,T] \times \RR^n \times A \to \RR^n\,,\quad \BV^0 : [0,T] \times \RR^n \times A \to \RR^{n \times n}, ,
\]
as
\[
\BD^0(t,\mathbf{X}^0(t),\alpha^0(t)) := \left (\mu^{1;0}(t,X^{1;0}(t),\alpha^{1;0}(t)), \dots, \mu^{n;0}(t,X^{n;0}(t),\alpha^{n;0}(t))\right )^T\, ,
\]
and
\[
\BV^0(t,\mathbf{X}^0(t),\alpha^0(t)) := diag [\sigma^{1;0}(t,X^{1;0}(t),\alpha^{1;0}(t)), \dots, \sigma^{n;0}(t,X^{n;0}(t),\alpha^{n;0}(t))]\, ,
\]
that is the matrix with $\sigma^{i;0}(t,x,a)$ entry on the diagonal and null off-diagonal.

Let us also denote $\mathbf{x}^0_0 = \left (x^{1;0}_0,\dots,x^{n;0}_0\right )$ and $\BW(t)= \left (W^{1}(t),\dots,W^{n}(t)\right )^T$. Hence, system \eqref{EQN:SysiIniz1} can be compactly rewritten as follows 
\begin{equation}\label{EQN:VectBan1}
\begin{cases}
d \mathbf{X}^0(t) &= \BD^0(t,\BX^0(t),\alpha^0(t)) dt + \BV^0(t,\BX^0(t),\alpha^0(t))d\BW(t) \, ,\\
\mathbf{X}^0(0) &= \mathbf{x}_0^0\, .
\end{cases}
\end{equation}

We will minimize the following functional
\begin{equation}\label{EQN:GenCP}
J(x,\alpha) = \mathbb E \int_0^{\hat{\tau}^{1}} \BL^0 \left (t,\BX^0 (t),\alpha^0(t) \right ) dt + \BG^0\left (\hat{\tau}^{1}, \mathbf{X}^0 (\hat{\tau}^{1}) \right ) \, ,\\
\end{equation}
where $\BL^0$ and $\BG^0$ are assumed to satisfy the following assumptions:

\begin{Assumptions}\label{ASS:2}
Let $\BL^{0}:[0,T] \times \RR^n \times A^0  \to \RR$ and $\BG^{0}:[0,T] \times \RR^n \to \RR$ be two measurable and continuous functions such that there exist two constants $K$, $k>0$ such that, for any $t \in [0,T]$, $x \in \RR^n$ and $a \in A^0$, it holds
\[
\begin{split}
&|\BL^0(t,x,a)| \leq K(1+ |x|^k + |a|^k)\, ,\\
&|\BG^0(t,x)| \leq K(1+|x|^k)\, .
\end{split}
\]
\end{Assumptions}

Let us underline that in the  cost functional defined by \eqref{EQN:GenCP}, the terminal time $\hat{\tau}^1$ is assumed to be triggered as soon as $\BX^0$ reaches a given boundary $v^0$. In particular, we assume the stopping boundary to be of the form
\[
v^0 = \left (v^{1;0},\dots,v^{n;0}\right )\, ,
\]
for some given constants $v^{i;0} \in \RR$, $i=1,\dots,n$. We thus denote by
\begin{equation}\label{EQN:Stop0}
\tau^{i;0} := T \wedge min \left \{ t \geq 0 \, : \, X^{i;0}(t) = v^{i;0}\, \right \}\,,\quad i=1,\dots,n ,
\end{equation}
the first time $X^{i;0}$ reaches the boundary $v^{i;0}$ and we set
\[
\hat{\tau}^{1} := \tau^{1;0} \wedge \dots \wedge \tau^{n;0}\, ,
\]
the first stopping time to happen. 

We stress that, in what follows we will denote by $\hat{\tau}$ the ordered stopping times. In particular,  $\hat{\tau}^1 \leq \dots \leq \hat{\tau}^n$, where $\hat{\tau}^k$ denotes the $k-th$ stopping time to happen. On the contrary, the notation $\tau^k$ indicates  that the stopping time has been triggered by the $k$-th node. In what follows, we will use the convention that, if $\hat{\tau}^1 = \tau^k$, then $\tau^j = T$, for $j \not = k$.

\begin{Remark}
From a practical point of view, we are considering a controller that aims at supervise $n$ different elements defininf a system, up to the first time one of the element of it exits from a given domain. From a financial perspective, each element represents a financial agent, while  the stopping time denotes its failure time. 
Hence,  a possible cost to be optimized, as we shall see in Section \ref{SEC:LQP1}, is to maximize the distance between the element/financial agent from the associated stopping/default boundary.
\end{Remark}

As briefly mentioned in the introduction, instead of stopping the overall control problem when the first stopping time is triggered, we assume that the system continues to evolve according to a (possibly) new dynamic. As to make an example, let us  consider the case of $\hat{\tau}^1 \equiv \hat{\tau}^{k;0}$, that is the first process to hit the stopping boundary is $X^{k;0}$. We thus set to $0$ the $k-$th component of $\BX^0$, then considering the new process
\[
\BX^{k}(t) = \left (X^{1;k}(t),\dots,X^{k-1;k}(t),0,X^{k+1;k}(t),\dots,X^{k;n}(t)\right )\,,
\]
with control given by
\[
\alpha^{k}(t) = \left (\alpha^{1;k}(t),\dots,\alpha^{k-1;k}(t),0,\alpha^{k+1;k}(t),\dots,\alpha^{k;n}(t)\right )\,,
\]
where the superscript $k$ denotes that the $k-$th component hit the stopping boundary and therefore has been set to 0. 

Then, we consider the $n-$dimensional system, for $t \in [\hat{\tau}^1,T]$, defined by
\[
\begin{cases}
dX^{i;k} (t) &= \mu^{i;k}\left (t,X^{i;k}(t),\alpha^{i;k}(t) \right ) dt + \sigma^{i;k}\left (t,X^{i;k}(t),\alpha^{i;k}(t)\right )dW^{i}(t)\, ,\\
X^{i;k}(\hat{\tau}^{1})&=X^{i;0}(\hat{\tau}^{1}) =: x^{i;k}\, ,\quad i = 1,\dots,k-1,k+1,\dots,n\, ,
\end{cases}
\]
where the coefficients $\mu^{i;k}$ and $\sigma^{i;k}$ satisfy assumptions \ref{ASS:1} and we have also set  $X^{k;k}(t) = 0$.

We therefore define $\BD^{k}:[\hat{\tau}^{1},T]\times \RR^{n} \times A \to \RR^{n}$ and $\BV^{k}:[\hat{\tau}^{1},T]\times \RR^{n} \times A \to \RR^{n \times n}$ as
\[
\begin{split}
\BD^k(t,\mathbf{X}^k(t),\alpha^k(t)) &:= \left (\mu^{1;k}(t,X^{1;k}(t),\alpha^{1;k}(t)), \dots, \mu^{n;k}(t,X^{n;k}(t),\alpha^{n;k}(t))\right )^T\, ,\\
\BV^k(t,\mathbf{X}^k(t),\alpha^k(t)) &:= diag [\sigma^{1;k}(t,X^{1;k}(t),\alpha^{1;k}(t)), \dots, \sigma^{n;k}(t,X^{n;k}(t),\alpha^{n;k}(t))]\, ,
\end{split}
\]
which allows us to rewrite the above system as
\begin{equation}\label{EQN:Sys2}
\begin{cases}
d \BX^{k}(t) = \BD^{k}(t,\BX^{k}(t),\alpha^k(t)) dt + \BV^{k}(t,\BX^{k}(t),\alpha^k(t))dW(t)\, , \quad t \geq \hat{\tau}^1\, ,\\
\BX^{k}(\tau^{k}) = \Phi^k(\tau^{k})\BX^0(\tau^{k}) =: \mathbf{x}^{k}\, ,
\end{cases}
\end{equation}
where $\Phi^k$ is the diagonal $n \times n$ matrix defined as
\[
\Phi^k(\tau^{k}) = diag\left [1,\dots,1,0,1,\dots,1\right ]\, ,
\]
 the null-entry being in the $k-$th position.

Then we minimize the following functional
\[
J^k(x,\alpha) = \mathbb E \int_{\hat{\tau}^{1}}^{\hat{\tau}^{2}} \BL^k \left (t,\BX^k (t),\alpha^k(t) \right ) dt + \BG^k\left (\hat{\tau}^{2}, \mathbf{X}^k (\hat{\tau}^{2}) \right ) \, ,\\
\]
where $\BL^k$ and $\BG^k$ are assumed to satisfy assumptions \ref{ASS:2}, while $\hat{\tau}^{2}$ is a stopping time triggered as soon as $\BX^k$ hits a defined boundary. In particular, we define the stopping boundary 
\[
v^{k} = \left (v^{1;k},\dots,v^{k-1;k},1,v^{k+1;k},\dots,v^{n;k}\right )\, ,\quad t \in [\hat{\tau}^{1},T]\, ,
\]
and, following the same scheme as before, we define by
\[
\tau^{i;k} := T \wedge min \left \{ t \geq \hat{\tau}^{1} \, : \, X^{i;k}(t) = v^{i;k}\, \right \}\,,\quad i=1,\dots,k-1,k+1,\dots,n ,
\]
the first time $X^{i;k}$ reaches the boundary $v^{i;k}$, denoting  
\[
\hat{\tau}^{2} := \tau^{1;k} \wedge \dots \wedge \tau^{k-1;k} \wedge \tau^{k+1;k} \wedge \dots \wedge \tau^{n;k}\, .
\]

It folllows that, considering for instance the case $\tau^{l;k}$ has been triggered by $X^{l;k}$, we have $\hat{\tau}^{2} \equiv \hat{\tau}^{l;k}$, meaning that $v^{l;k}$ has been hit. Iteratively proceeding, we consequently define

\[
\begin{split}
\BX^{(k,l)}(t) &= \left (X^{1;(k,l)}(t),\dots,X^{k-1;(k,l)}(t),0, X^{k+1;(k,l)}(t),\dots,\right .\\
&\left .\,\quad ,X^{l-1;(k,l)}(t),0,X^{l+1;(k,l)}(t),\dots,X^{n;(k,l)}(t)\right )^T\,,
\end{split}
\]

again assuming $\BX^{(k,l)}(t)$ evolves according to a system  as in \eqref{EQN:VectBan1}, and so on
until either no nodes are left or the terminal time $T$ is reached. 


As mentioned above, one of the major novelty of the present work consists in not assuming the knowledge 
of the stopping times order. From a mathematical point of view, the latter implies that we have to consider
all the possible combinations of such {\it critical points} during a given time interval $[0,T]$. 
Let us note that this is in fact the natural setting to work with having in mind the modelling of concrete
scenarios, as happens, e.g., concerning possible multiple failures happening within a system of interconnected
banks.

Therefore, in what follows we are going to denote by 
 $C_{n,k}$ the combinations of $k$ elements from a set of $n$, while $\pi^k \in C_{n,k}$ stands for one of those {\it element}. 
Hence, exploiting the  notation introduced above, we define the process $\BX = \left (\BX(t)\right )_{t \in [0,T]}$ as
\begin{equation}\label{EQN:BankInd1}
\BX(t) =  \BX^0(t) \Ind{t < \hat{\tau}^{1}} +\sum_{k =1}^{n-1} \sum_{\pi^k \in C_{n,k}} \BX^{\pi^k}(t) \Ind{\tau^{\pi^k} < t < \hat{\tau}^{k+1}} \, ,
\end{equation}
where each $ \BX^{\pi^k}(t)$ is defined as above and, consequently, the  
the global control reads as follow
\begin{equation}\label{EQN:Control1}
\alpha(t) =  \alpha^0(t) \Ind{t < \hat{\tau}^{1}} +\sum_{k =0}^{n-1} \sum_{\pi^k \in C_{n,k}} \alpha^{\pi^k}(t) \Ind{\tau^{\pi^k} < t < \hat{\tau}^{k+1}} \, .
\end{equation}

\begin{Remark}
Let us underlined within the setting defined so far,
each stopping time $\hat{\tau}^k$ depends on previously triggered stopping times $\tau^{\pi^j}$, $j=1,\dots,k-1$. As a consequence, also the solution $X^{\pi^k}$ in \eqref{EQN:BankInd1} depends on triggered stopping times as well as on their order. 
To simplify notation, we have avoided to explicitly write such dependencies,  defining for short
\[
\hat{\tau}^k := \hat{\tau}^k(\hat{\tau}^1,\dots,\hat{\tau}^{k-1})\,.
\]
\end{Remark}

By equation \eqref{EQN:BankInd1} we have that the dynamic for $\BX$ is given by
\begin{equation}\label{EQN:BankDynInd1}
d \BX (t) = \BD(t,\BX(t),\alpha(t)) dt + \BV(t,\BX(t),\alpha(t)) d W(t)\, ,
\end{equation}
where, according to the above introduced notation, we have defined 
\begin{equation}\label{EQN:DVolInd}
\begin{split}
\BD(t,\BX(t),\alpha(t)) &= \BD^0(t,\BX^0(t),\alpha^0(t)) \Ind{t < \hat{\tau}^{1}} + \\
&+\sum_{k =1}^{n-1} \sum_{\pi^k \in C_{n,k}} \BD^{\pi^k}(t,\BX^{\pi^k}(t),\alpha^{\pi^k}(t))  \Ind{\tau^{\pi^k} < t < \hat{\tau}^{k+1}} \, ,\\
\BV(t,\BX(t),\alpha(t)) &= \BV^0(t,\BX^0(t),\alpha^0(t))\Ind{t < \hat{\tau}^{1}} +\\
&+ \sum_{k =1}^{n-1} \sum_{\pi^k \in C_{n,k}} \BV^{\pi^k}(t,\BX^{\pi^k}(t),\alpha^{\pi^k}(t)) \Ind{\tau^{\pi^k} < t < \hat{\tau}^{k+1}} \, ,
\end{split}
\end{equation}
aiming at minimizing the 
following functional
\begin{equation}\label{EQN:GenCP1}
\begin{split}
J(x,\alpha) &:= \mathbb E  \int_0^{\hat{\tau}^{n}} \BL \left (t,\BX (t), \alpha(t) \right ) dt + \BG\left (\hat{\tau}^{n}, \mathbf{X} \left (\hat{\tau}^{n}\right ) \right ) \, .\\
\end{split}
\end{equation}
$\BL$ and $\BG$ being defined as

\[
\begin{split}
\BL(t,\BX(t),\alpha(t)) &= \BL^0(t,\BX^0(t),\alpha^0(t))  \Ind{t < \hat{\tau}^{1}} +\\
&+ \sum_{k =1}^{n-1} \sum_{\pi^k \in C_{n,k}} \BL^{\pi^k}(t,\BX^{\pi^k}(t),\alpha^{\pi^k}(t))  \Ind{\tau^{\pi^k} < t < \hat{\tau}^{k+1}} \, ,\\
\BG\left (\hat{\tau}^{n}, \mathbf{X} \left (\hat{\tau}^{n}\right ) \right ) &=\BG^0(\hat{\tau}^1,\BX^0(\hat{\tau}^1))  \Ind{\hat{\tau}^1 \leq T} +\\
&+ \sum_{k =1}^{n} \sum_{\pi^k \in C_{n,k}} \BG^{\pi^k}(\tau^{\pi^k},\BX^{\pi^k}(\tau^{\pi^k}))  \Ind{\tau^{\pi^k} < T \leq  \hat{\tau}^{k+1}}  \, .
\end{split}
\]

\begin{Remark}
It is worth to mention that we are considering the sums stated above as to be done
over all possible combinations, hence implying we are not considering components' order, namely considering $X^{(k,l)} = X^{(l,k)}$. Dropping such an assumption
implies that the sums in equations \eqref{EQN:BankInd1}--\eqref{EQN:Control1}--\eqref{EQN:DVolInd} have to be considered 
 over the disposition $D_{n,k}$.
\end{Remark}

In what follows we shall give an example of the theory developed so far, as to better clarify our approach as well as its concrete applicability.

\begin{Example}
Let us consider the case of a system constituted by just $n=2$ components. Then equation \eqref{EQN:BankInd1} becomes
\[
\BX(t) = \BX^0(t) \Ind{t < \hat{\tau}^{1}} + \BX^{1}(t) \Ind{\tau^1 < t < \hat{\tau}^{2}} + \BX^{2}(t) \Ind{\tau^2 < t < \hat{\tau}^{2}}\, ,
\]
where $\BX^{0}(t)$, resp. $\BX^{1}(t)$, resp. $\BX^{2}(t)$, denotes the dynamics in case neither $1$ nor $2$ has hit the stopping boundary, resp. $1$ has, resp. $2$ has.

Then, denoting by $\alpha^0(t)$, $\alpha^1(t)$ and $\alpha^2(t)$ the respective associated controls, we have that the functional \eqref{EQN:GenCP1} reads
\[
\begin{split}
J(x,\alpha) &:= \mathbb E  \int_0^{\hat{\tau}^{1}} \BL^0 \left (t,\BX^0 (t),\alpha^0(t) \right ) dt + \BG^0\left (\hat{\tau}^{1}, \mathbf{X}^0 (\hat{\tau}^{1}) \right ) + \\
&+ \mathbb E  \int_{\tau^1}^{\hat{\tau}^{2}} \BL^1 \left (t,\BX^1 (t), \alpha^1(t) \right ) dt + \BG^1\left (\hat{\tau}^{2}, \mathbf{X}^1 (\hat{\tau}^{2}) \right ) +\\
&+ \mathbb E \int_{\tau^2}^{\hat{\tau}^{2}} \BL^2 \left (t,\BX^2 (t),\alpha^2(t) \right ) dt + \BG^2\left (\hat{\tau}^{2}, \mathbf{X}^2 (\hat{\tau}^{2}) \right ) \, .
\end{split}
\]
\end{Example}

\subsection{A necessary maximum principle}\label{SEC:NMP}

The main issue in solving the optimal control problem defined in Section \ref{SEC:GS} consists in solving  a series of connected optimal problems, each of which may depends on previous ones. Moreover, we do not assume to have an a priori knowledge about the stopping times' order. 

To overcome such issues, we consider a {\it backward approach}. In particular,
we first solve the last control problem, then proceeding with the penultimate, and so on, until the first one, via  backward induction. 
Let us underline that assuming the perfect knowledge of the stopping times' order would imply a simplification of the backward scheme, because of the need  to solve only  $n$ control problems, then saving us to take into account all the combinations. Nevertheless in one case as in the other, the backward procedure runs analogously.

Aiming at deriving a 
\textit{global maximum principle}, in what follows we denote 
 by $\partial_x$ the partial derivative w.r.t. the space variable $x \in \RR^n$ and by $\partial_a$ the partial derivative w.r.t. the control $a \in A^n$.
Moreover we assume

\begin{Assumptions}\label{ASS:CF}
\begin{description}
\item[(i)] For any $\pi^k \in C_{n,k}$, $k=1,\dots,n$, it holds that $\BD^{\pi^k}$ and $\BV^{\pi^k}$ are continuously differentiable w.r.t. to both $x \in \RR^n$ and to $a \in A$. Furthermore, there exists a constant $C_1>0$ such that for any $t \in [0,T]$, $x \in \RR^n$ and $a \in A$, it holds
\[
\begin{split}
|\partial_x \BD^{\pi^k}(t,x,a)| + |\partial_a \BD^{\pi^k}(t,x,a)| \leq C_1\, ,\\
|\partial_x \BV^{\pi^k}(t,x,a)| + |\partial_a \BV^{\pi^k}(t,x,a)| \leq C_1\, .\\
\end{split}
\]

\item[(ii)] For any $\pi^k \in C_{n,k}$, $k=1,\dots,n$, it holds that $\BL^{\pi^k}$, resp. $\BG^{\pi^k}$, is continuously differentiable w.r.t. to both $x \in \RR^n$ and  $a \in A^n$, resp. only w.r.t.  $x \in \RR^n$. Furthermore, there exists a constant $C_2>0$ such that for any $t \in [0,T]$, $x \in \RR^n$ and $a \in A^n$, it holds
\[
\begin{split}
&|\partial_x \BL^{\pi^k}(t,x,a)| + |\partial_a \BL^{\pi^k}(t,x,a)| \leq C_2(1+|x|+|a|)\, ,\\
&|\partial_x \BG^{\pi^k}(t,x)| \leq C_2\,.\\
\end{split}
\]
\end{description}
\end{Assumptions}

We thus have the following result.

\begin{Theorem}\label{THM:SMP}[Necessary Maximum Principle]
Let assumptions \ref{ASS:1}--\ref{ASS:2}--\ref{ASS:CF} hold and let $\left (\bar{\BX},\bar{\alpha}\right )$ be an optimal pair for the problem \eqref{EQN:BankDynInd1}--\eqref{EQN:GenCP1}, then it holds
\begin{equation}\label{EQN:OptimalSol}
\begin{split}
\langle \partial_a \mathit{H}(t,\bar{\BX}(t),\bar{\alpha}(t),\bar{Y}(t),\bar{Z}(t)), \left (\bar{\alpha}(t) -\tilde{\alpha}\right )\rangle &\leq 0 \, ,\\
& \mbox{a.e. }\, t \in [0,\hat{\tau}^{n}]\, ,\mathbb{P}-a.s\, , \forall \tilde{\alpha} \in A\, ,
\end{split}
\end{equation}
equivalently
\[
\bar{\alpha}(t) = \arg \min_{\tilde{\alpha}\in A} \mathit{H}(t,\bar{\BX}(t),\tilde{\alpha}(t),Y(t),Z(t))\, ,
\]

where the pair $(Y(t),Z(t))$ solves the following dual backward equation
\[
\begin{split}
Y(t)&= Y^0(t)  \Ind{t < \hat{\tau}^{1}} + \sum_{k =1}^{n-1} \sum_{\pi^k \in C_{n,k}} Y^{\pi^k}(t) \Ind{\tau^{\pi^k} < t < \hat{\tau}^{k+1}} \, ,\\
Z(t)&= Z^0(t)  \Ind{t < \hat{\tau}^{1}} + \sum_{k =1}^{n-1} \sum_{\pi^k \in C_{n,k}} Z^{\pi^k}(t) \Ind{\tau^{\pi^k} < t < \hat{\tau}^{k+1}} \, ,\\
\end{split}
\]
the pairs  $(Y^{\pi^k}(t),Z^{\pi^k}(t))$ being solutions of the following system of interconnected BSDEs
\begin{equation}\label{EQN:SMP}
\begin{split}
&
\begin{cases}
- d Y^{\pi^{n-1}} (t) &= \partial_x \mathit{H}^{\pi^{n-1}} (t,\BX^{\pi^{n-1}}(t),\alpha^{\pi^{n-1}}(t),Y^{\pi^{n-1}}(t),Z^{\pi^{n-1}}(t)) dt - Z^{\pi^{n-1}} dW(t)\, ,\\
Y^{\pi^{n-1}}(\hat{\tau}^{n}) &= \partial_x \BG^{\pi^{n-1}}(\hat{\tau}^{n},\BX^{\pi^{n-1}}(\hat{\tau}^{n}))\, ,
\end{cases}\\
\\
&
\begin{cases}
- d Y^{\pi^k} (t) &= \partial_x \mathit{H}^{\pi^k} (t,\BX^{\pi^k}(t),\alpha^{\pi^k}(t),Y^{\pi^k}(t),Z^{\pi^k}(t)) dt - Z^{\pi^k} dW(t)\, ,\\
Y^{\pi^k}(\hat{\tau}^{k+1}) &= \partial_x \BG^{\pi^k}(\hat{\tau}^{k+1},\BX^{\pi^k}(\hat{\tau}^{k+1})) + \bar{Y}^{k+1}(\hat{\tau}^{k+1})\, ,
\end{cases}\\
\\
&
\begin{cases}
- d Y^{0} (t) = \partial_x \mathit{H}^{0} (t,\BX^{0}(t),\alpha^{0}(t),Y^{0}(t),Z^{0}(t)) dt - Z^{0} dW(t)\, ,\\
Y^{0}(\tau_{1}) = \partial_x \BG^0(\tau_{1},\BX^{0}(\tau_{1})) + \bar{Y}^1(\tau_1)\, ,
\end{cases}
\end{split}
\end{equation}
having denoted by
\[
\bar{Y}^{\pi^{k+1}}(\hat{\tau}^{k+1}) := \sum_{\pi^{k+1} \in C_{n,k+1}} Y^{\pi^{k+1}}\left (\tau^{\pi^{k+1}}\right ) \Ind{\hat{\tau}^{k+1} = \tau^{\pi^{k+1}}}\, ,
\]
where $\mathit{H}^{\pi^k}$ is the \textit{generalized Hamiltonian} 
\[
\mathit{H}^{\pi^k}: [0,T] \times \RR^{n} \times A \times \RR^{n} \times \RR^{n \times n} \to \RR\, ,
\]
defined as
\begin{equation}\label{EQN:Ham0SC}
\begin{split}
\mathit{H}^{\pi^k}(t,x^{\pi^k},a^{\pi^k},y^{\pi^k},z^{\pi^k}) &:= \BD^{\pi^k}(t,x^{\pi^k},a^{\pi^k}) \cdot y^{\pi^k} +\\
&+ Tr[(\BV^{\pi^k}(t,x^{\pi^k},a^{\pi^k}))^* z^{\pi^k}] + \BL^{\pi^k}(t,x^{\pi^k},a^{\pi^k})\, ,
\end{split}
\end{equation}
and $\mathit{H}$ represents the global \textit{generalized Hamiltonian} defined as
\[
\begin{split}
\mathit{H}(t,x,a,y,z)&=\mathit{H}^0(t,x,a,y,z)  \Ind{t <  \hat{\tau}^{1}} + \\
&+\sum_{k =1}^{n-1} \sum_{\pi^k \in C_{n,k}} \mathit{H}^{\pi^k}(t,x,a,y,z) \Ind{\tau^{\pi^k} < t < \hat{\tau}^{k+1}} \, .\\
\end{split}
\]
\end{Theorem}
\begin{Remark}\label{REM:Spieg}
Before entering into details about proving Theorem \ref{THM:SMP}, let us underline some of its characteristics.
In particular, here the main idea is to find a solution iteratively acting backward in time.
Therefore, starting from the very last control problem, namely the case where a single node is left into the system, we consider a standard maximum principle. Indeed, $Y^{\pi^{n-1}}$ in \eqref{EQN:SMP} represents a classical dual BSDE form associated to the standard \textit{stochastic maximum principle}, see, e.g., \cite[Th. 3.2]{Yon}.
Then, we can consider  the second last control problem. A this point, a naive tentative to obtain a {\it global} solution, could
be to first solve such penultimate problem to then gluing together the obtained solutions.
Nevertheless, such a method only produces a  a suboptimal solution. 
Instead, the right approach, similarly to what happens applying the standard \textit{dynamic programming principle}, consists in treating the solution to the last control problem as the terminal cost for the subsequent (second last) control problem, and so on for the remaining ones.

It follows that, in deriving the global optimal solution, one considers the cost coming from future evolution of the system. Mathematically, this is clearly expressed by the terminal condition $Y^{\pi^k}$ the equation \eqref{EQN:SMP} is endowed with. Therefore the solution scheme resulting in a global connection of all the control problems we have to consider, from the very last of them and then backward to the first one.

\end{Remark}
\begin{proof}[Proof. [Necessary Maximum Principle]]
We proceed according to a backward induction technique. In particular,
for $t_0> \hat{\tau}^{n-1}$ the proof follows from the standard \textit{stochastic necessary maximum principle}, see, e.g.. \cite[Th. 3.2]{Yon}.
Then we consider the case of $ \hat{\tau}^{n-2} < t_0 < \hat{\tau}^{n-1}$, and we define
\[
\bar{\alpha} :=
\begin{cases}
\bar{\alpha}^{\pi^{n-2}}(t) & t_0<t<\hat{\tau}^{n-2}\, ,\\
\bar{\alpha}^{\pi^{n-1}}(t) & \hat{\tau}^{n-2}<t<\hat{\tau}^{n-1}\, .
\end{cases}
\]
to be the optimal control, $\alpha$ being another admissible control and further setting $\alpha^h$ as
\[
\alpha^h := \bar{\alpha} + h \alpha\, ,\quad h >0\, .
\]
Since in the present case the cost functional reads as follow
\[
\begin{split}
J(x,\alpha) &:= \mathbb E \int_{t_0}^{\hat{\tau}^{n-1}} \BL^{\pi^{n-2}} \left (t,\BX^{\pi^{n-2}} (t),\alpha^{\pi^{n-2}}(t) \right ) dt + \BG^{\pi^{n-2}}\left (\hat{\tau}^{n-1}, \mathbf{X} (\hat{\tau}^{n-1}) \right ) +\\
&+\sum_{\pi^{n-1} \in C_{n,n-1}} \mathbb E \int_{\tau^{\pi^{n-1}}}^{\hat{\tau}^{n}} \BL^{\pi^{n-1}} \left (t,\BX^{\pi^{n-1}} (t),\alpha^{\pi^{n-1}}(t) \right ) dt +\\
&+ \BG^{\pi^{n-1}}\left (\hat{\tau}^{n}, \mathbf{X} (\hat{\tau}^{n}) \right )\, ,
\end{split}
\]
we can choose 
$\alpha = \bar{\alpha}- \tilde{\alpha}$, $\tilde{\alpha} \in A$. Then, by the optimality of $\bar{\alpha}$ and via a standard variational argument, see, e.g., \cite{BCD,Mou,Yon}, we have
\[
 J(x,\bar{\alpha}) - J(x,\alpha^h) \leq 0\, ,
\]
which implies
\[
\lim_{h \to 0} \frac{ J(x,\bar{\alpha}) - J(x,\alpha^h)}{h} \leq 0\, .
\]

In what follows, for the sake of clarity, we will denote by $\BX_{\alpha}$ the solution $\BX$ with control $\alpha$. Thus, from the optimality of $\bar{\alpha}$, we have
\begin{equation}\label{EQN:VarIneqB}
\begin{split}
&\mathbb E \int_{t_0}^{\hat{\tau}^{n-1}} \BL^{\pi^{n-2}} \left (t,\bar{\BX}_{\bar{\alpha}}^{\pi^{n-2}} (t),\bar{\alpha}^{\pi^{n-2}}(t) \right ) dt + \BG^{\pi^{n-2}}\left (\hat{\tau}^{n-1}, \bar{\mathbf{X}}_{\bar{\alpha}}^{\pi^{n-2}} (\hat{\tau}^{n-1})\right )+\\
&+\sum_{\pi^{n-1} \in C_{n,n-1}}\int_{\tau^{\pi^{n-1}}}^{\hat{\tau}^{n}} \BL^{\pi^{n-1}} \left (t,\bar{\BX}_{\bar{\alpha}}^{\pi^{n-1}} (t), \bar{\alpha}^{\pi^{n-1}}(t) \right ) dt + \BG^{\pi^{n-1}}\left (\hat{\tau}^{n}, \bar{\mathbf{X}}_{\bar{\alpha}}^{\pi^{n-1}} (\hat{\tau}^{n})\right ) \leq \\
&\leq \mathbb E \int_{t_0}^{\hat{\tau}^{n-1}} \BL^{\pi^{n-2}} \left (t,\bar{\BX}_{\alpha^h}^{\pi^{n-2}} (t),\bar{\alpha}^{\pi^{n-2}}(t) \right ) dt + \BG^{\pi^{n-2}}\left (\hat{\tau}^{n-1}, \bar{\mathbf{X}}_{\alpha^h}^{\pi^{n-2}} (\hat{\tau}^{n-1})\right )+\\
&+\sum_{\pi^{n-1} \in C_{n,n-1}}\int_{\tau^{\pi^{n-1}}}^{\hat{\tau}^{n}} \BL^{\pi^{n-1}} \left (t,\bar{\BX}_{\alpha^h}^{\pi^{n-1}} (t),\bar{\alpha}^{\pi^{n-1}}(t) \right ) dt + \BG^{\pi^{n-1}}\left (\hat{\tau}^{n}, \bar{\mathbf{X}}_{\alpha^h}^{\pi^{n-1}} (\hat{\tau}^{n})\right ) \, .\\
\end{split}
\end{equation}
Then, for any $\alpha \in \mathcal{A}$, by 
\eqref{EQN:VarIneqB}, we obtain 
\begin{equation}\label{EQN:VarIneqB2}
\begin{split}
&\mathbb E \int_{t_0}^{\hat{\tau}^{n-1}}  \partial_x \BL^{\pi^{n-2}} \left (t,\bar{\BX}_{\bar{\alpha}}^{\pi^{n-2}} (t) ,\bar{\alpha}^{\pi^{n-2}}(t)\right )\mathfrak{Z}^{\pi^{n-2}}(t)  dt+\\
&+ \partial_x \BG^{\pi^{n-2}}\left (\hat{\tau}^{n-1}, \bar{\mathbf{X}}_{\bar{\alpha}}^{\pi^{n-2}} (\hat{\tau}^{n-1})\right )\mathfrak{Z}^{\pi^{n-2}}(\hat{\tau}^{n-1}) + \\
&+\sum_{\pi^{n-1} \in C_{n,n-1}}\mathbb E \int_{\tau^{\pi^{n-1}}}^{\hat{\tau}^{n}} \partial_x \BL^{\pi^{n-1}} \left (t,\bar{\BX}_{\bar{\alpha}}^{\pi^{n-1}} (t),\bar{\alpha}^{\pi^{n-1}}(t) \right )\mathfrak{Z}^{\pi^{n-1}}(t)  dt +\\
&+\sum_{\pi^{n-1} \in C_{n,n-1}} \partial_x \BG^{\pi^{n-1}}\left (\hat{\tau}^{n}, \bar{\mathbf{X}}_{\bar{\alpha}}^{\pi^{n-1}} (\tau^{\pi^{n}})\right )\mathfrak{Z}^{\pi^{n-1}}(\hat{\tau}^{n}) \leq 0
\end{split}
\end{equation}
where $\mathfrak{Z}^{\pi^{n-1}}$ and $\mathfrak{Z}^{\pi^{n-2}}$ solve the first variation process

\begin{align*}
&
\begin{cases}
d \mathfrak{Z}^{\pi^{n-1}} (t) &= \partial_x \BD^{\pi^{n-1}}\left (t,\bar{\BX}^{\pi^{n-1}}(t), \alpha^{\pi^{n-1}}(t)\right )\mathfrak{Z}^{\pi^{n-1}} (t)dt +\\
&+ \partial_a \BD^{\pi^{n-1}}\left (t,\bar{\BX}^{\pi^{n-1}}(t), \alpha^{\pi^{n-1}}(t)\right )\alpha^{\pi^{n-1}}(t)dt +\\
&+ \partial_x \BV^{\pi^{n-1}}\left (t,\bar{\BX}^{\pi^{n-1}}(t), \alpha^{\pi^{n-1}}(t)\right )\mathfrak{Z}^{\pi^{n-1}} (t) dW(t)+\\
&+ \partial_a \BV^{\pi^{n-1}}\left (t,\bar{\BX}^{\pi^{n-1}}(t), \alpha^{\pi^{n-1}}(t)\right )\alpha^{\pi^{n-1}}(t) dW(t)\, ,\\
\mathfrak{Z}^{\pi^{n-1}} (\hat{\tau}^{n-1}) &= \bar{\mathfrak{Z}}^{\pi^{n-2}} (\hat{\tau}^{n-1})\, ,\quad t \in [\hat{\tau}^{n-1},\hat{\tau}^{n}]\, ,
\end{cases}
\\
&
\begin{cases}
d \mathfrak{Z}^{\pi^{n-2}} (t) &= \partial_x \BD^{\pi^{n-2}}\left (t,\bar{\BX}^{\pi^{n-2}}(t), \alpha^{\pi^{n-2}}(t)\right )\mathfrak{Z}^{\pi^{n-2}} (t)dt+\\
&+ \partial_a \BD^{\pi^{n-2}}\left (t,\bar{\BX}^{\pi^{n-2}}(t), \alpha^{\pi^{n-2}}(t)\right )\alpha^{\pi^{n-2}}(t) dt + \\
&+\partial_x \BV^{\pi^{n-2}}\left (t,\bar{\BX}^{\pi^{n-2}}(t), \alpha^{\pi^{n-2}}(t)\right )\mathfrak{Z}^{\pi^{n-2}} (t)dW(t) +\\
&+ \partial_a \BV^{\pi^{n-2}}\left (t,\bar{\BX}^{\pi^{n-2}}(t), \alpha^{\pi^{n-2}}(t)\right )\alpha^{\pi^{n-2}}(t) dW(t)\, ,\\
\mathfrak{Z}^{\pi^{n-2}} (t_0) &= 0\, ,\quad t \in [t_0,\hat{\tau}^{n-1}]\, .
\end{cases}
\end{align*}

Applying It\^{o} formula to $Y^{\pi^{n-2}} \cdot \mathfrak{Z}^{\pi^{n-2}}$, we have

\begin{equation}\label{EQN:ItoMP1}
\begin{split}
&\mathbb{E}\left ( \partial_x \BG^{\pi^{n-2}}\left (\hat{\tau}^{n-1},\BX^{\pi^{n-2}}(\hat{\tau}^{n-1})\right ) + \bar{Y}^{n-1}(\hat{\tau}^{n-1})\right )\cdot \mathfrak{Z}^{\pi^{n-2}}(\hat{\tau}^{n-1}) =\\
&= \mathbb{E}\, Y^{\pi^{n-2}}(\hat{\tau}^{n-1}) \cdot \mathfrak{Z}^{\pi^{n-1}}(\hat{\tau}^{n-1}) =\\
&= -\mathbb{E} \int_{t_0}^{\tau^{\pi^{n-1}}} \left (\partial_x \mathit{H}^{\pi^{n-2}} (t,\BX^{\pi^{n-2}}(t),\alpha^{\pi^{n-2}}(t),Y^{\pi^{n-2}}(t),Z^{\pi^{n-2}}(t)) dt\right ) \cdot \mathfrak{Z}^{\pi^{n-2}}(t) dt+\\
&+\mathbb{E} \int_{t_0}^{\hat{\tau}^{n-1}} \left (\partial_x \BD^{\pi^{n-2}}\left (t,\bar{\BX}^{\pi^{n-2}}(t), \alpha^{\pi^{n-2}}(t)\right )\mathfrak{Z}^{\pi^{n-2}} (t)\right ) \cdot Y^{\pi^{n-2}}(t) dt +\\
&+\mathbb{E} \int_{t_0}^{\hat{\tau}^{n-1}}  \left ( \partial_a \BD^{\pi^{n-2}}\left (t,\bar{\BX}^{\pi^{n-2}}(t), \alpha^{\pi^{n-2}}(t)\right )\alpha^{\pi^{n-2}}(t)\right ) \cdot Y^{\pi^{n-2}}(t) dt +\\
&+\mathbb{E} \int_{t_0}^{\hat{\tau}^{n-1}} \left (\partial_x \BV^{\pi^{n-2}}\left (t,\bar{\BX}^{\pi^{n-2}}(t), \alpha^{\pi^{n-2}}(t)\right )\mathfrak{Z}^{\pi^{n-2}} (t)\right )\cdot Z^{\pi^{n-2}} (t)dt +\\
&+\mathbb{E} \int_{t_0}^{\hat{\tau}^{n-1}} \left (\partial_a \BV^{\pi^{n-2}}\left (t,\bar{\BX}^{\pi^{n-2}}(t), \alpha^{\pi^{n-2}}(t)\right )\alpha^{\pi^{n-2}}(t)\right )\cdot Z^{\pi^{n-2}} (t)dt = \\
&=- \mathbb{E}\int_{t_0}^{\hat{\tau}^{n-1}} \partial_x \BL^{\pi^{n-2}}\left (t,\BX^{\pi^{n-2}}(t),\alpha^{\pi^{n-2}}(t)\right )\mathfrak{Z}^{\pi^{n-2}}(t) \\
&+\mathbb{E}\int_{t_0}^{\hat{\tau}^{n-1}} \left ( \partial_a \BD^{\pi^{n-2}}\left (t,\bar{\BX}^{\pi^{n-2}}(t), \alpha^{\pi^{n-2}}(t)\right )Y^{\pi^{n-2}}(t)\right ) \cdot \alpha^{\pi^{n-2}}(t)dt+\\
&+ \mathbb{E}\int_{t_0}^{\hat{\tau}^{n-1}} \left (\partial_a \BV^{\pi^{n-2}}\left (t,\bar{\BX}^{\pi^{n-2}}(t), \alpha^{\pi^{n-2}}(t)\right )Z^{\pi^{n-2}}(t)\right ) \cdot \alpha^{\pi^{n-2}}(t)dt+\\\, ,
\end{split}
\end{equation}

and similarly for $Y^{\pi^{n-1}} \cdot \mathfrak{Z}^{\pi^{n-1}}$, we obtain
\begin{equation}\label{EQN:ItoMP2}
\begin{split}
&\mathbb{E}\left ( \partial_x \BG^{\pi^{n-1}}(\hat{\tau}^{n},\BX^{\pi^{n-1}}(\hat{\tau}^{n}))\right )\cdot \mathfrak{Z}^{\pi^{n-1}}(\hat{\tau}^{n}) = \mathbb{E}\, Y^{\pi^{n-1}}(\hat{\tau}^{n}) \cdot \mathfrak{Z}^{\pi^{n-1}}(\hat{\tau}^{n}) =\\
&=\mathbb{E} Y^{\pi^{n-1}}(\tau^{\pi^{n-1}})\cdot \mathfrak{Z}^{\pi^{n-1}}(\tau^{\pi^{n-1}})+\\
&-\mathbb{E}\int_{\tau^{\pi^{n-1}}}^{\hat{\tau}^{n}} \partial_x \BL^{\pi^{n-1}}(t,\BX^{\pi^{n-1}}(t),\bar{\alpha}^{\pi^{n-1}}(t))\mathfrak{Z}^{\pi^{n-1}}(t) dt+\\
&+\mathbb{E}\int_{\tau^{\pi^{n-1}}}^{\hat{\tau}^{n}} \left ( \partial_a \BD^{\pi^{n-1}}\left (t,\bar{\BX}^{\pi^{n-1}}(t), \alpha^{\pi^{n-1}}(t)\right )Y^{\pi^{n-1}}(t)\right ) \cdot \alpha^{\pi^{n-1}}(t)dt+\\
&+\mathbb{E}\int_{\tau^{\pi^{n-1}}}^{\hat{\tau}^{n}} \left ( \partial_a \BV^{\pi^{n-1}}\left (t,\bar{\BX}^{\pi^{n-1}}(t), \alpha^{\pi^{n-1}}(t)\right )Z^{\pi^{n-1}}(t)\right ) \cdot \alpha^{\pi^{n-1}}(t)dt \, .
\end{split}
\end{equation}

Exploiting equation \eqref{EQN:VarIneqB2}, together with equations \eqref{EQN:ItoMP1}--\eqref{EQN:ItoMP2}, we thus have

{\footnotesize
\[
\begin{split}
&\int_{t_0}^{\hat{\tau}^{n-1}}\left (\partial_\alpha \mathit{H}^{\pi^{n-2}}(t,\bar{\BX}^{\pi^{n-2}}(t),\bar{\alpha}^{\pi^{n-2}}(t),\bar{Y}^{\pi^{n-2}}(t),\bar{Z}^{\pi^{n-2}}(t))\right )\alpha^{\pi^{n-2}}(t) dt +\\
&+ \sum_{\pi^{n-1}\in C_{n,n-1}}\int_{\tau^{\pi^{n-1}}}^{\hat{\tau}^{n}}\left (\partial_\alpha \mathit{H}^{\pi^{n-1}}(t,\bar{\BX}^{\pi^{n-1}}(t),\bar{\alpha}^{\pi^{n-1}}(t),\bar{Y}^{\pi^{n-1}}(t),\bar{Z}^{\pi^{n-1}}(t))\right )\alpha^{\pi^{n-1}}(t) dt \leq 0\, ,
\end{split}
\]
}

for all $\alpha =\bar{\alpha}- \tilde{\alpha}$, and thus we eventually obtain, for $t_0 > \hat{\tau}^{n-2}$
\[
\partial_\alpha \mathit{H}(t,\bar{\BX}(t),\bar{\alpha}(t),\bar{Y}(t),\bar{Z}(t)) ( \bar{\alpha}(t)-\tilde{\alpha}) \leq 0\, ,\quad \mbox{a.e. }\, t \in [t_0,\hat{\tau}^{n}]\, ,\mathbb{P}-a.s\, , \forall \tilde{\alpha} \in A\, ,
\]
which is the desired local form for optimality \eqref{EQN:OptimalSol}.
Analogously proceeding via backward induction, we derive that the same results also hold 
for any $\pi^k \in C_{n,k}$, hence obtaining the system \eqref{EQN:SMP} and concluding the proof.
\end{proof}

\subsection{A sufficient maximum principle}\label{SEC:SMP}

In this section we  consider a generalization of the classical sufficient maximum principle, see, e.g., \cite[Th. 6.4.6]{PhaBook}, for the present setting of interconnected multiple optimal control problems with random terminal time.
To this end, we assume
\begin{Assumptions}\label{ASS:S}
For any $\pi^k \in C_{n,k}$ the derivative w.r.t. $x$ of $\BD$, $\BV$ and $\BL$ are continuous and there exists a constant $L^a >0$ such that, for any $a_1$, $a_2 \in A$,
\[
\begin{split}
&|\BD^{\pi^k}(t,x,a_1)-\BD^{\pi^k}(t,x,a_2)| + |\BV^{\pi^k}(t,x,a_1)-\BV^{\pi^k}(t,x,a_2)|+\\
&+|\BL^{\pi^k}(t,x,a_1)-\BL^{\pi^k}(t,x,a_2)| \leq L^a|a_1-a_2|\, .
\end{split}
\]
\end{Assumptions}

\begin{Theorem}[Sufficient maximum principle]\label{THM:SSMP}
Let \ref{ASS:1}--\ref{ASS:2}--\ref{ASS:CF}--\ref{ASS:S} hold, let $\left (Y,Z\right )$ be the solution to  the dual BSDE \ref{EQN:SMP}, and suppose the following  conditions hold true 
\begin{description}
\item[(i)] the maps $x \mapsto G^{\pi^k}(x)$ are convex for any $\pi^k$;
\item[(ii)] the maps $(x,a)\mapsto H^{\pi^k}\left (x,a,Y^{\pi^k},Z^{\pi^k}\right )$ are convex for a.e. $t \in [0,T]$ and for any $\pi^k$;
\item[(iii)] for a.e. $t \in [0,T]$ and $\mathbb{P}-$a.s. it holds
\[
\bar{\alpha}^{\pi^k}(t) = \arg\min_{\tilde{\alpha}^{\pi^k} \in \mathcal{A}^{\pi^k}} H^{\pi^k}\left (t,X^{\pi^k}(t),\tilde{\alpha}(t),Y^{\pi^k},Z^{\pi^k}\right )\,,
\]
\end{description}
then $\left (\bar{\alpha},\bar{\mathbf{X}}\right )$ is an optimal pair for the problem \eqref{EQN:BankDynInd1}--\eqref{EQN:GenCP1}.
\end{Theorem}
\begin{proof}
Let us proceed as in the proof of Theorem \ref{THM:SMP}, namely via backward induction. For $t_0> \hat{\tau}^{n-1}$ the proof follows from the standard \textit{sufficient stochastic maximum principle}, see, e.g., \cite[Th. 5.2]{Yon}.

Let us thus then consider the case of $ \hat{\tau}^{n-2} < t_0 < \hat{\tau}^{n-1}$,  denoting by $\Delta X^{\pi^k}(t):= \bar{X}^{\pi^k}(t) - X^{\pi^k}(t)$ and, for the sake of clarity, by using similar notations  for any other function.

The convexity of $G^{\pi^{n-1}}$, together with the terminal condition 
\[
Y^{\pi^{n-1}}(\hat{\tau}^n) = \partial_x \mathbf{G}^{\pi^{n-1}}(\hat{\tau}^n,\bar{\mathbf{X}}^{\pi^{n-1}}(\hat{\tau}^n))\,,
\]
yields
\begin{equation}\label{EQN:ConvG}
\begin{split}
&\mathbb{E}\Delta \mathbf{G}^{\pi^{n-1}}(\hat{\tau}^n,\bar{\mathbf{X}}^{\pi^{n-1}}(\hat{\tau}^n)) \leq\\
&\leq \mathbb{E}\left [\Delta X^{\pi^{n-1}}(\hat{\tau}^n) \partial_x\mathbf{G}^{\pi^{n-1}}(\hat{\tau}^n,\bar{\mathbf{X}}^{\pi^{n-1}}(\hat{\tau}^n))\right ] = \mathbb{E}\left [\Delta X^{\pi^{n-1}}(\hat{\tau}^n) Y^{\pi^{n-1}}(\hat{\tau}^n)\right ] \,.
\end{split}
\end{equation}

Applying the It\^{o}-formula to $\Delta X^{\pi^{n-1}} Y^{\pi^{n-1}}(\hat{\tau}^n)$, we obtain
\begin{equation}\label{EQN:ItoG}
\begin{split}
&\mathbb{E}\left [\Delta X^{\pi^{n-1}}(\hat{\tau}^n)) Y^{\pi^{n-1}}(\hat{\tau}^n)\right ] = \mathbb{E}\left [\Delta X^{\pi^{n-1}}(\hat{\tau}^{n-1}) Y^{\pi^{n-1}}(\hat{\tau}^{n-1})\right ] + \\
&+\mathbb{E}\int_{\hat{\tau}^{n-1}}^{\hat{\tau}^n} \Delta X^{\pi^{n-1}}(t) dY^{\pi^{n-1}}(t) + \mathbb{E}\int_{\hat{\tau}^{n-1}}^{\hat{\tau}^n} Y^{\pi^{n-1}}(t) d \Delta X^{\pi^{n-1}}(t) +\\
&+ \mathbb{E}\int_{\hat{\tau}^{n-1}}^{\hat{\tau}^n} Tr\left [\Delta \Sigma^{\pi^{n-1}}(t,\mathbf{X}^{\pi^{n-1}}(t),\alpha^{\pi^{n-1}}(t))Z^{\pi^{n-1}}(t) \right ]dt=\\
&= \mathbb{E}\left [\Delta X^{\pi^{n-1}}(\hat{\tau}^{n-1}) Y^{\pi^{n-1}}(\hat{\tau}^{n-1})\right ] +\\
& -\mathbb{E}\int_{\hat{\tau}^{n-1}}^{\hat{\tau}^n} \Delta X^{\pi^{n-1}}(t)\partial_x H^{\pi^{n-1}}\left (t,\bar{X}^{\pi^{n-1}}(t),\bar{\alpha}^{\pi^{n-1}}(t),Y^{\pi^{n-1}}(t),Z^{\pi^{n-1}}(t)\right )dt +\\
&+\mathbb{E}\int_{\hat{\tau}^{n-1}}^{\hat{\tau}^n}\left ( \Delta \mathbf{B}^{\pi^{n-1}}(t) Y^{\pi^{n-1}}(t) + \Delta \mathbf{\Sigma}^{\pi^{n-1}}(t) Z^{\pi^{n-1}}(t) \right )dt\,.
\end{split}
\end{equation}

Similarly, from the convexity of the Hamiltonian, we also have 
\begin{equation}\label{EQN:CHam}
\begin{split}
&\mathbb{E} \int_{\hat{\tau}^{n-1}}^{\hat{\tau}^n} \left [\Delta \mathbf{L}^{\pi^{n-1}}\left (t,\bar{\mathbf{X}}^{\pi^{n-1}}(t),\bar{\alpha}^{\pi^{n-1}}(t)\right ) \right ]dt =\\
&= \mathbb{E} \int_{\hat{\tau}^{n-1}}^{\hat{\tau}^n} \left [\Delta H^{\pi^{n-1}}\left (t,\bar{\mathbf{X}}^{\pi^{n-1}}(t),\bar{\alpha}^{\pi^{n-1}}(t),Y^{\pi^{n-1}}(t),Z^{\pi^{n-1}}(t)\right ) \right ]dt + \\
&-\mathbb{E}\int_{\hat{\tau}^{n-1}}^{\hat{\tau}^n}\left ( \Delta \mathbf{B}^{\pi^{n-1}}(t) Y^{\pi^{n-1}}(t) + \Delta \mathbf{\Sigma}^{\pi^{n-1}}(t) Z^{\pi^{n-1}}(t) \right )dt \leq \\
&\leq \mathbb{E} \int_{\hat{\tau}^{n-1}}^{\hat{\tau}^n} \left [\Delta X^{\pi^{n-1}}(t) \partial_x H^{\pi^{n-1}}\left (t,\bar{\mathbf{X}}^{\pi^{n-1}}(t),\bar{\alpha}^{\pi^{n-1}}(t),Y^{\pi^{n-1}}(t),Z^{\pi^{n-1}}(t)\right ) \right ]dt + \\
&-\mathbb{E}\int_{\hat{\tau}^{n-1}}^{\hat{\tau}^n}\left ( \Delta \mathbf{B}^{\pi^{n-1}}(t) Y^{\pi^{n-1}}(t) + \Delta \mathbf{\Sigma}^{\pi^{n-1}}(t) Z^{\pi^{n-1}}(t) \right )dt \,,
\end{split}
\end{equation}
so that, for any $\pi^{n-1}$, by combining equations \eqref{EQN:ConvG}--\eqref{EQN:ItoG}--\eqref{EQN:CHam}, we derive
\begin{equation}\label{EQN:DisN1}
\begin{split}
&\mathbb{E} \int_{\hat{\tau}^{n-1}}^{\hat{\tau}^n} \left [\Delta \mathbf{L}^{\pi^{n-1}}\left (t,\bar{\mathbf{X}}^{\pi^{n-1}}(t),\bar{\alpha}^{\pi^{n-1}}(t)\right ) \right ]dt + \mathbb{E}\Delta \mathbf{G}^{\pi^{n-1}}(\hat{\tau}^n,\bar{\mathbf{X}}^{\pi^{n-1}}(\hat{\tau}^n)) \leq \\
&\leq \mathbb{E}\left [\Delta X^{\pi^{n-1}}(\hat{\tau}^{n-1}) Y^{\pi^{n-1}}(\hat{\tau}^{n-1})\right ]\,.
\end{split}
\end{equation}
Analogously, for $t_0 \in [\hat{\tau}^{n-2},\hat{\tau}^{n-1}]$, and since 
\[
Y^{\pi^{n-2}}(\hat{\tau}^{n-1}) = \partial_x \mathbf{G}^{\pi^{n-2}}(\hat{\tau}^{n-1},\bar{\mathbf{X}}^{\pi^{n-2}}(\hat{\tau}^{n-1})) + \bar{Y}^{\pi^{n-1}}(\hat{\tau}^{n-1})\,,
\]
together with the convexity of $\mathbf{G}^{\pi^{n-2}}$,  we have
\[
\begin{split}
&\mathbb{E}\Delta \mathbf{G}^{\pi^{n-2}}(\hat{\tau}^{n-1},\bar{\mathbf{X}}^{\pi^{n-2}}(\hat{\tau}^{n-1})) \leq\\
&\leq \mathbb{E}\left [\Delta X^{\pi^{n-2}}(\hat{\tau}^{n-1}) \partial_x\mathbf{G}^{\pi^{n-2}}(\hat{\tau}^{n-1},\bar{\mathbf{X}}^{\pi^{n-2}}(\hat{\tau}^{n-1}))\right ] =\\
&= \mathbb{E}\left [\Delta X^{\pi^{n-2}}(\hat{\tau}^{n-1}) Y^{\pi^{n-2}}(\hat{\tau}^{n-1}) - \Delta X^{\pi^{n-2}}(\hat{\tau}^{n-1})\bar{Y}^{\pi^{n-1}}(\hat{\tau}^{n-1}) \right ] \,.
\end{split}
\]
Similar computations also give us

\begin{equation}\label{EQN:DisN2}
\begin{split}
&\mathbb{E} \int_{\hat{\tau}^{n-2}}^{\hat{\tau}^{n-1}} \left [\Delta \mathbf{L}^{\pi^{n-2}}\left (t,\bar{\mathbf{X}}^{\pi^{n-2}}(t),\bar{\alpha}^{\pi^{n-2}}(t)\right ) \right ]dt + \mathbb{E}\Delta \mathbf{G}^{\pi^{n-2}}(\hat{\tau}^{n-1},\bar{\mathbf{X}}^{\pi^{n-2}}(\hat{\tau}^{n-1})) \leq \\
&\leq - \mathbb{E} \Delta X^{\pi^{n-2}}(\hat{\tau}^{n-1})\bar{Y}^{\pi^{n-1}}(\hat{\tau}^{n-1}) \,,
\end{split}
\end{equation}
so that, for $t_0 \in [\hat{\tau}^{n-2},\hat{\tau}^{n-1}]$, by equations \eqref{EQN:DisN1}--\eqref{EQN:DisN1}, we infer that 
\begin{equation}\label{EQN:ResFin}
\begin{split}
J(t_0,x,\bar{\alpha})-J(t_0,x,\alpha) &:= \mathbb E \int_{t_0}^{\hat{\tau}^{n-1}} \BL^{\pi^{n-2}} \left (t,\mathbf{\bar{X}}^{\pi^{n-2}} (t),\bar{\alpha}^{\pi^{n-2}}(t) \right ) dt +\\
&+ \mathbb{E}\BG^{\pi^{n-2}}\left (\hat{\tau}^{n-1}, \mathbf{\bar{X}}^{\pi^{n-2}} (\hat{\tau}^{n-1}) \right ) +\\
&+\sum_{\pi^{n-1} \in C_{n,n-1}} \mathbb E \int_{\tau^{\pi^{n-1}}}^{\hat{\tau}^{n}} \BL^{\pi^{n-1}} \left (t,\mathbf{\bar{X}}^{\pi^{n-1}} (t),\bar{\alpha}^{\pi^{n-1}}(t) \right ) dt +\\
&+\BG^{\pi^{n-1}}\left (\hat{\tau}^{n}, \mathbf{\bar{X}}^{\pi^{n-1}} (\hat{\tau}^{n}) \right )+\\
&+\mathbb E \int_{t_0}^{\hat{\tau}^{n-1}} \BL^{\pi^{n-2}} \left (t,\BX^{\pi^{n-2}} (t),\alpha^{\pi^{n-2}}(t) \right ) dt +\\
&+ \mathbb{E}\BG^{\pi^{n-2}}\left (\hat{\tau}^{n-1}, \mathbf{X} (\hat{\tau}^{n-1}) \right ) +\\
&+\sum_{\pi^{n-1} \in C_{n,n-1}} \mathbb E \int_{\tau^{\pi^{n-1}}}^{\hat{\tau}^{n}}\BL^{\pi^{n-1}} \left (t,\BX^{\pi^{n-1}} (t),\alpha^{\pi^{n-1}}(t) \right ) dt +\\
&+ \BG^{\pi^{n-1}}\left (\hat{\tau}^{n}, \mathbf{X} (\hat{\tau}^{n}) \right ) \leq 0\,,
\end{split}
\end{equation}
which implies that
\[
J(t_0,x,\bar{\alpha}) \leq J(t_0,x,\alpha)\,,
\]
and the optimality of $\left (\bar{\alpha},\bar{\mathbf{X}}\right )$.

Proceeding backward, previously exploited arguments allow us to show the same results for any $\pi^k \in C_{n,k}$, hence ending the proof.
\end{proof}

\section{The linear--quadratic problem}\label{SEC:LQP1}

In the present section we  consider a particular case for the control problem stated in Section \ref{SEC:NMP}--\ref{SEC:SMP}. 
In particular, we will assume that the dynamic of the state equation is linear in both the space and the control variable. Moreover, we impose that the control enters (linearly) only in the drift and that the cost functional is quadratic and of a specific form.
More precisely, let us first consider  $\mathbf{\mu}^0(t)$ as the $n \times n$ matrix defined as follows
\[
\mathbf{\mu}^0(t) := diag [\mu^{1;0}(t), \dots, \mu^{n;0}(t)]\,,
\]
that is the matrix with $\mu^{i;0}(t)$ entry on the diagonal and null off-diagonal,  $\mu^{i;0}:[0,T] \to \RR$ being a deterministic and bounded function of the time. Also let 
\[
\mathbf{b}^0(t) = (b^{1;0}(t),\dots,b^{n;0}(t))^T\,,
\]
where again $b^{i;0}:[0,T] \to \RR$ is a deterministic and bounded function of time. Then we set 
\begin{equation}\label{EQN:LineD}
\BD^0(t,\BX^0(t),\alpha(t)) = \mathbf{\mu}^0(t)\BX^0(t) +\mathbf{b}^0(t) + \alpha(t)\, .
\end{equation}

Let us also define the $n \times n$ matrix $\BV^0$, to be independent of the control, as follows
\begin{equation}\label{EQN:BV0}
\BV^0(t,\BX^0(t)) := \left( \begin{array}{ccc}
 \sigma^{1;0}(t) X^{1;0}(t) + \nu^{1;0}(t) & 0 & 0 \\
 0 & \ddots & 0  \\
 0 & 0 & \sigma^{n;0}(t) X^{n;0}(t) +\nu^{n;0}(t)  \\
 \end{array} \right)\, ,
\end{equation}
$\sigma^{i;0}$, $\nu^{i;0} : [0,T] \to \RR$ being deterministic and bounded function of time.

Same assumptions of linearity holds for any other coefficients $\BD^{\pi^k}$ and $\BV^{\pi^k}$, so that, using the same notation introduced along previous sections, we consider the system 
\begin{equation}\label{EQN:BankDynInd2}
d \BX (t) = \BD(t,\BX(t),\alpha(t)) dt + \BV(t,\BX(t)) d W(t)\, ,
\end{equation}
where  both the drift and the volatility coefficients are now assumed to be linear.
In the present (particular) setting, both the running and the terminal cost are assumed to be suitable quadratic weighted averages of the distance from the stopping boundaries, namely we set
\begin{equation}\label{EQN:Av1}
\begin{split}
\BL^{\pi^k}(t,x,a) &= \sum_{i=1}^n\left ( \gamma_i^{\pi^k} \frac{|x_i-v^{i;\pi^k}|^2}{2} + \frac{1}{2} |a^{i;\pi^k}|^2\right )\, ,\\
\BG^{\pi^k}(t,x) &= \sum_{i=1}^n \gamma_i^{\pi^k} \frac{|x_i-v^{i;\pi^k}|^2}{2}\, ,
\end{split}
\end{equation}
for some given weights $\gamma^{\pi^k}$ such that
\[
\gamma^{\pi^k} = (\gamma^{\pi^k}_1,\dots,\gamma^{\pi^k}_n)^T\, .
\]


\begin{Remark}
From a financial perspective, converting the minimization problem into a maximization one, the above cost functional can be seen as a financial supervisor, such as the one introduced in \cite{Cap,CDPP}, aiming at lending money to each node (e.g., a bank, a financial player, an institution, etc.) in the system 
to avert it from the corresponding  (default) boundary.
Continuing the financial interpretation, different weights $\gamma$ can be used to assign to any node a relative importance. This allows to establish a hierarchy of (financial) relevance within the system, resulting in a priority scale related to the systemic (monetary) importance took on by each node.
As to give an example,  in \cite{CDPP} a systematic procedure has been derived to obtain the overall importance of any node in a financial network.
\end{Remark}

In what follows, we derive a set of Riccati BSDEs to provide the global optimal control in  feedback form. 
For the sake of notation clarity,  we  denote by $X^{k;-k}(t)$ the dynamics when  only the $k-$th node is left. Similarly, $X^{k;-(k,l)}(t)$, resp. $X^{l;-(k,l)}(t)$, denotes the evolution of the node $k$, resp. of the node $l$, when this pair $(k,l)$ survives. 
Analogously, we will make use of a componentwise notation, namely $X^{i;-k}$ will denote the $i-$th component of th $n-$dimensional vector $X^{-k}$.
According to such a notation, we have the following

\begin{Theorem}\label{THM:VT}
The optimal control problem \eqref{EQN:BankDynInd2}, with associated costs given by \eqref{EQN:Av1}, has an optimal feedback control solution given by
\[
\bar{\alpha}(t) = P(t) \BX(t)+\varphi(t)  \, ,
\]
where  $P$ and $\varphi$ are defined as follows
\begin{equation}\label{EQN:PePhi}
\begin{split} 
P(t) &= P^0(t)\Ind{t <\hat{\tau}^{1}} + \sum_{k =1}^{n-1} \sum_{\pi^k \in C_{n,k}} P^{\pi^k}(t)\Ind{\tau^{\pi^k} < t < \hat{\tau}^{k+1}} \, ,\\
\varphi(t) &= \varphi^0(t)\Ind{t < \hat{\tau}^{1}} + \sum_{k =1}^{n-1} \sum_{\pi^k \in C_{n,k}} \varphi^{\pi^k}(t)\Ind{\tau^{\pi^k} < t <\hat{\tau}^{k+1}} \, ,
\end{split}
\end{equation}
 $P^{\pi^k}$ and $\varphi^{\pi^k}$ being solution to the following recursive system of  Riccati BSDEs

{\footnotesize
\[
\begin{split}
&
\begin{cases}
-dP^{\pi^{n-1}}(t) &= \left (\left (P^{\pi^{n-1}}(t)\right )^2  + \left (\sigma^{\pi^{n-1}}(t)\right )^2 P^{\pi^{n-1}}(t) +2 Z^{\pi^{n-1};P} (t) \sigma^{\pi^{n-1}}(t) - 1\right )dt +\\
&- Z^{\pi^{n-1};P}(t) dW^{\pi^{n-1}}(t)\, ,\\
P^{\pi^{n-1}}(\hat{\tau}^{n}) &= 1\, ,
\end{cases}\\
\\
&
\begin{cases}
-d\varphi^{\pi^{n-1}}(t) &= \left ((P^{\pi^{n-1}}(t)-\mu^{\pi^{n-1}}(t)) \varphi^{\pi^{n-1}}(t)+ \sigma^{\pi^{n-1}}(t) Z^{\pi^{n-1};\varphi}(t) - h^{\pi^{n-1}}\left (P(t),v(t)\right )\right ) dt+\\
& - Z^{\pi^{n-1};\varphi}(t) dW^{\pi^{n-1}}(t)\, ,\\
\varphi^{\pi^{n-1}}(\hat{\tau}^{n}) &= - v^{\pi^{n-1}}(\hat{\tau}^{n}) \, ,
\end{cases}\\
\\
&
\begin{cases}
-dP^{\pi^k}(t) &= \left (-P^{\pi^k}(t)^2  + (\sigma^{\pi^k})^2 P^{\pi^k}(t)\right )dt +\\
&+\left (Z^{\pi^k;P}_j (t) \sigma^{\pi^k}(t) - \gamma^{\pi^k}\right )dt - Z^{\pi^k;P}(t) dW^{\pi^k(t)}\, ,\\
P^{\pi^k}(\hat{\tau}^{n-1}) &=\gamma^{\pi^k} - \sum_{\pi^{k+1} \in C_{n,k+1}} P^{\pi^{k+1}}(\hat{\tau}^{n-1}) \Ind{\hat{\tau}^{k+1} = \tau^{\pi^{k+1}}} \, ,\\
\end{cases}\\
\\
&
\begin{cases}
-d\varphi^{\pi^k}(t) &= \left ((\mu^{\pi^k}(t)-P^{\pi^k}(t)) \varphi^{\pi^k}(t)+ \sigma^{\pi^k}(t) Z^{\pi^k;\varphi}(t)\right )dt+\\
& -h^{\pi^k}(P^{\pi^k}(t), v^{\pi^k}(t)) dt - Z^{\pi^k;\varphi}(t) dW^{\pi^k}(t)\, ,\\
\varphi(\hat{\tau}^{n-1}) &= - \gamma^{} v^{\hat{\tau}^{n-1}}(\hat{\tau}^{n-1}) + \sum_{\pi^{k+1} \in C_{n,k+1}} \varphi^{\pi^{k+1}}(\hat{\tau}^{n-1})\Ind{\hat{\tau}^{k+1} = \tau^{\pi^{k+1}}}\, ,
\end{cases}\\
\\
&
\begin{cases}
-dP^{0)}(t) &= \left (-(P^{0}(t))^2  + (\sigma^{0}(t))^2 P^{0}(t) + Z^{0;P} (t) \sigma^{0}(t) - \gamma^{0}\right )dt +\\
&- Z^{0;P}(t) dW(t)\, ,\\
P^{0}(\hat{\tau}^{1}) &=\gamma^{0}- \sum_{\pi \in C_{n,1}} P^{1}(\hat{\tau}^{1})\Ind{\hat{\tau}^{1} = \tau^{\pi}} \, ,\\
\end{cases}\\
\\
&
\begin{cases}
-d\varphi^{0}(t) &= \left ((\mu^{0)}-P^{0}(t)) \varphi^{0}(t)+ \sigma^{0}(t) Z^{0;\varphi}(t) - \gamma^{0} v^{0}(t)\right ) dt+\\
& - Z^{0;\varphi}(t) dW(t)\, ,\\
\varphi^{0}(\hat{\tau}^{1}) &= \gamma^{0}v^{0}(\hat{\tau}^{1}) - \sum_{\pi \in C_{n,1}} \varphi^{1}(\hat{\tau}^{1})\Ind{\hat{\tau}^{1} = \tau^{\pi}} \, .
\end{cases}
\end{split}
\]
}

\end{Theorem}
\begin{proof}

Let us thus first consider the last control problem, recalling that $\mathit{H}^{-k}(t,x,a,y,z)$ is the \textit{generalized Hamiltonian} defined in \eqref{EQN:Ham0SC}, where $\BD^{-k}$, resp. $\BV^{-k}$, resp. $\BL^{-k}$, is given in equation \eqref{EQN:LineD}, resp. equation \eqref{EQN:BV0}, resp. equation \eqref{EQN:Av1}. An application of the \textit{stochastic maximum principle}, see Theorems \ref{THM:SMP}--\ref{THM:SSMP}, leads us to consider the following adjoint BSDE
\begin{equation}\label{EQN:Back1}
\begin{split}
Y^{-k} (t) = \partial_x \BG^{-k}\left (\BX^{-k}(\hat{\tau}^{n})\right )&+\int_t^{\hat{\tau}^{n}} \partial_x \mathit{H}^{-k}\left (\BX^{-k}(s),\alpha^{-k}(s),Y^{-k}(s),Z^{-k}(s)\right ) ds +\\
&- \int_t^{\hat{\tau}^{n}} Z^{-k}(s) dW(s)\, , \quad t \in [0,\hat{\tau}^{n}]\,,
\end{split}
\end{equation}
$Y^{-k}$ being  a $n-$dimensional vector, whereas $Z^{-k}$ is a $n \times n$ matrix whose $(i,j)-$entry is denoted by $Z^{-k}_{i,j}$.
Then, considering the particular form for $\BD^{-k}(t,x,a)$, $\BV^{-k}(t,x)$, $\BL^{-k}(t,x,a)$ and $\BG^{-k}(t,x)$, in equations \eqref{EQN:LineD}--\eqref{EQN:BV0}--\eqref{EQN:Av1}, we have
\[
\begin{split}
\partial_{x_k} \mathit{H}^{-k}(t,x,a,y,z) &= \mu^{-k;k}(t)  y_k + \sigma^{k:-k} z_{k,k} + \gamma_k^{-k}|x_k-v^{k;-k}|\, ,\\
\partial_{x_k} \BG^{-k}(t,x) &= \gamma_k^{-k}| x_k-v^{k;-k}|\, ,
\end{split}
\]
and
\[
\partial_{x_i} \mathit{H}^{-k}(t,x,a,y,z) = 0 = \partial_{x_i} \BG^{-k}(t,x) \, ,\quad \mbox{ if } \, i \not = k\, ,
\]
where $\partial_{x_i}$ denotes the derivative w.r.t. the $i-$th component of $x \in \RR^n$.

Thus we have that the $k-$th component of the BSDE \eqref{EQN:Back1} now reads

{\footnotesize
\begin{equation}\label{EQN:Back1Exp}
\begin{split}
Y^{k;-k} (t) &= \gamma_k^{-k} X^{k;-k}(\hat{\tau}^{n}) - \gamma_k^{-k} v^{k;-k}(\hat{\tau}^{n}) +\\
&+ \int_t^{\hat{\tau}^{n}} \left ( \mu^{k;-k}(s) Y^{k;-k}(s) + \sigma^{k;-k} (s) Z^{-k}_{k,k}(s) + \gamma_k^{-k} X^{k;-k}(s) - \gamma_k^{-k}v^{k;-k}(s) \right ) ds +\\
&- \int_t^{\hat{\tau}^{n}} Z^{-k}_{k,k}(s) dW^k(s)\, , \quad t \in [0,\hat{\tau}^{n}]\,.
\end{split}
\end{equation}
}

Analogously, we have that the second last control problem is associated to the following system of BSDEs
\begin{equation}\label{EQN:Back2Exp}
\begin{split}
Y^{i;-(k,l)} (t) &= \gamma_i^{-(k,l)} X^{i;-(k,l)}(\hat{\tau}^{n-1}) - \gamma^{-(k,l)} v^{i;-(k,l)}(\tau^{\pi^{n-1}}) + \bar{Y}^{i;n-1}(\hat{\tau}^{n-1})\\
&+ \int_t^{\hat{\tau}^{n-1}} \left ( \mu^{i;-(k,l)}(s) Y^{i;-(k,l)}(s) + \sum_{j=k}^l \sigma^{j;-(k,l)}(s) Z^{-(k,l)}_{j,j}(s)\right ) ds +\\
&+ \int_t^{\hat{\tau}^{n-1}} \left ( \gamma_i^{-(k,l)} X^{i;-(k,l)}(s) - \gamma_i^{-(k,l)}v^{i;-(k,l)}(s) \right ) ds +\\
&-\sum_{j=k}^l \int_t^{\hat{\tau}^{n-1}} Z_{i,j}^{-(k,l)}(s) dW^j(s)\, , \quad t \in [0,\tau^{\pi^{n-1}}]\,, \quad i=k,l\, ,
\end{split}
\end{equation}
and so on for any $\pi^k$, until we reach the first control problem with associated the following BSDEs system
\begin{equation}\label{EQN:Back1Exp0L}
\begin{split}
Y^{i;0} (t) &= \gamma_i^{0} X^{i;0}(\hat{\tau}^{1}) - \gamma_i^0 v^0(\hat{\tau}^{1}) + \bar{Y}^{i;1}(\hat{\tau}^{1}) +\\
&+ \int_t^{\hat{\tau}^{1}} \left ( \mu^{i;0}(s) Y^{i;0}(s) + \sum_{j=1}^n \sigma^{j;0}(s) Z^{0}_{j,j}(s) + \gamma_i^{0} X^{i;0}(s) - \gamma_i^0 v^0(s) \right ) ds +\\
&-\sum_{j=1}^n \int_t^{\hat{\tau}^{1}} Z^{0}_{i,j}(s) dW^j(s)\, , \quad t \in [0,\tau^0]\,, \quad i=1,\dots,n\, .
\end{split}
\end{equation}
Therefore, for $t \in [0,\hat{\tau}^{n}]$, we are left with the minimization problem for
\[
\begin{split}
J(x,t) &:= \mathbb E_t \int_t^{\hat{\tau}^{n}}\left ( |X^{k;-k}(s) - v^{k;-k}(s)|^2  + \frac{1}{2}| \alpha^{k;-k}(s)|^2\right ) ds + \\
&+ \mathbb E_t |X^{k;-k}(\hat{\tau}^{1}) - v^{k;-k}(\hat{\tau}^{1})|^2 \, .
\end{split}
\]
Exploiting Theorem \ref{THM:SMP}, we have that, on the interval $[\tau^{\pi^{n-1}},\hat{\tau}^{n}]$, the above control problem is associated to the following forward--backward system
\begin{equation}\label{EQN:FBSDEL}
\begin{cases}
d X^{k;-k} (t) &= \left (\mu^{k;-k}(t) X^{k;-k}(t) +b^{k;-k}(t)+ \alpha^{k;-k}(t)\right ) dt +\\
&+\left ( \sigma^{k;-k}(t) X^{k;-k}(t)+ \nu^{k;-k}(t)\right ) d W^k(t)\, ,\\
X^{k;-k} (\tau^{\pi^{n-1}}) &= X^{k;n-1} (\tau^{\pi^{n-1}})\, ,\\
-dY^{k;-k} (t) &= \left ( \mu^{k;-k}(t) Y^{k;-k}(t) + \sigma^{k;-k}(t) Z^{-k}_{k,k}(t) + X^{k;-k}(t) - v^{k;-k}(t)\right ) dt +\\
&-  Z^{k;-k}_{k,k}(t) dW^k(t)\, ,\\
Y^{k;-k} (\hat{\tau}^{n}) &= X^{k;-k}(\hat{\tau}^{n})-v^{k;-k}(\hat{\tau}^{n})\, .
\end{cases}
\end{equation}

In what follows, for the sake of brevity, we will drop the index $(k;-k)$. Therefore,  until otherwise specified, we will write $X$ instead of $X^{k;-k}$,  and similarly for any other coefficients. We also recall that system \eqref{EQN:FBSDEL} has to be solved for any $k=1,\dots,n$.

We thus guess the solution of the backward component $Y$ in equation \eqref{EQN:FBSDEL} to be of the form
\begin{equation}\label{EQN:GuessBL}
-Y(t) = P(t) X(t) - \varphi(t) \, ,
\end{equation}
for $P$ and $\varphi$ two $\RR-$valued processes to be determined.

Notice that in standard cases, that is when the coefficients are not random or the terminal time is deterministic, $P$ and $\varphi$ solve a backward ODE, while
in the present case, because of the terminal time randomness, $P$ and $\varphi$ will solve a BSDE. 

Let us thus assume that $(P(t),Z^P(t))$ is the solution to
\begin{equation}\label{EQN:PGen}
-dP(t) = F^P(t) dt - Z^P(t) dW(t)\, ,\quad P(\hat{\tau}^{n}) = 1\, ,
\end{equation}
and that $(\varphi(t),Z^\varphi(t))$ solves
\begin{equation}\label{EQN:PhiGen}
-d\varphi(t) = F^\varphi(t) dt - Z^\varphi(t) dW^j(t)\, ,\quad \varphi(\hat{\tau}^{n}) = - v(\hat{\tau}^{n}) \, .
\end{equation}

From the first order condition, namely  $\partial_a H(t,x,a,y,z) = 0$, we have that the optimal control is given by
\begin{equation}\label{EQN:OPtAL}
\bar{\alpha} = -Y(t) = P(t) X(t) - \varphi(t) \, .
\end{equation}

An application of It\^{o} formula yields

{\footnotesize
\begin{equation}\label{EQN:Anstaz22}
\begin{split}
&\left ( \mu(t) Y(t) + \sigma(t) Z(t) + X (t)-v(t)\right ) dt - Z(t) dW(t) = -dY(t) = d(P(t) X(t))- d \varphi(t) =\\[1.5ex]
&= \left (- F^{P} (t) X(t) + P(t) \mu(t) X(t) + P(t) \alpha(t) + Z^{P} (t) \sigma(t) X(t) + Z^P(t) \nu(t)+P(t)b(t)+F^{\varphi}(t)\right ) dt+\\[1.5ex]
&+\left ( Z^{P} (t) X(t) + P(t) \sigma(t) X(t) + P(t) \nu(t) - Z^{\varphi}_j(t)\right ) dW(t) =\\[1.5ex]
&=\left ( -F^{P} (t) +  P(t) \mu(t) + Z^{P} (t) \sigma (t)\right )  X(t) dt +  P(t) \alpha(t) dt +\left ( Z^P(t) \nu(t)+P(t)b(t)+F^{\varphi}(t)\right ) dt+\\[1.5ex]
&+ \left (Z^{P} (t) +P(t) \sigma(t)\right ) X(t) dW(t)+ \left (P(t) \nu(t)-Z^{\varphi}(t)\right ) dW(t) \, .
\end{split}
\end{equation}
}
Therefore,  equating the left hand side and the right hand side of equation \eqref{EQN:Anstaz22}, we derive
\begin{equation}\label{EQN:Anstaz222}
-Z(t) = \left (Z^{P} (t) +P(t) \sigma(t)\right ) X(t) + \left (P(t) \nu(t)-Z^{\varphi}(t)\right )  \, ,\\
\end{equation}
moreover, by substituting  equation \eqref{EQN:Anstaz222} into the left hand side of equation \eqref{EQN:Anstaz22},  exploiting  the first order optimality condition \eqref{EQN:OPtAL}, and equating again the left hand side and the right hand side of equation \eqref{EQN:Anstaz22}, we obtain

{\footnotesize
\begin{equation}\label{EQN:SysdiX}
\begin{split}
&\left ( \mu(t)P(t) - \sigma(t)Z^{P} (t) - \sigma^2(t) P(t) + 1\right ) X(t) - \left ( \mu(t) \varphi(t) + \sigma(t) P(t) \nu(t) -\sigma(t) Z^\varphi (t) + v(t)\right ) =\\[1.5ex]
&=\left ( -F^{P} (t) +  P(t) \mu(t) + Z^{P} (t) \sigma (t) + P^2(t)\right )  X(t) +\left ( Z^P(t) \nu(t)+P(t)b(t)+F^{\varphi}(t) - P(t) \varphi(t)\right )\, .
\end{split}
\end{equation}
}

Since equation 	\eqref{EQN:SysdiX} has to  hold for any $X(t)$, we have
\begin{equation}\label{EQN:Sys1LM2}
\mu(t) P(t) - \sigma(t) Z^{P} (t) - \sigma^2(t) P(t)+1 = -F^{P} (t) +  P(t) \mu(t) + Z^{P} (t)\sigma(t) + P(t)^2\, ,
\end{equation}

which, after some computations,  leads to
\begin{equation}\label{EQN:Sys1LMA2}
F^{P} (t) = P(t)^2  + \sigma^2(t) P(t) +2 Z^{P} (t) \sigma(t) - 1  \, .
\end{equation}

Similarly, we also have that
\begin{equation}\label{EQN:Sys1LMA2Phi}
F^{\varphi}(t) = (P(t) - \mu(t)) \varphi(t)+ \sigma(t) Z^{\varphi}(t) - v(t)- \sigma(t)\nu(t)P(t)-Z^P(t)\nu(t)-P(t)b(t)\, ,
\end{equation}
hence using the  particular form for the generator $F^P$, resp. of $F^\varphi$, stated in equation \eqref{EQN:Sys1LMA2}, resp. in equation \eqref{EQN:Sys1LMA2Phi}, in equation \eqref{EQN:PGen}, resp. in equation \eqref{EQN:PhiGen}, and reintroducing, for the sake of clarity, the index $k$ , the last optimal control $\bar{\alpha}^{k;-k}(t)$ reads as follow
\[
\bar{\alpha}^{k;-k}(t) = P^{k;-k}(t) X^{k;-k}(t) - \varphi^{k;-k}(t)\, ,
\]
 $P^{k;-k}(t)$ and $\varphi^{k;-k}(t)$ being solutions to the BSDEs

{\footnotesize
\begin{equation}\label{EQN:SolPVLast}
\begin{cases}
-dP^{k;-k}(t) &= \left (\left (P^{k;-k}(t)\right )^2  + \left (\sigma^{k;-k}(t)\right )^2 P^{k;-k}(t) +2 Z^{-k;P}_{k,k} (t) \sigma^{k;-k}(t) - 1\right )dt +\\
&- Z^{-k;P}_{k,k}(t) dW^k(t)\, ,\\
P^{k;-k}(\hat{\tau}^{n}) &= 1\, ,
\end{cases}
\end{equation}
}

{\footnotesize
\begin{equation}\label{EQN:SolPVLastPhi}
\begin{cases}
-d\varphi^{k;-k}(t) &= \left ((P^{k;-k}(t)-\mu^{k;-k}(t)) \varphi^{k;-k}(t)+ \sigma^{k;-k}(t) Z^{-k;\varphi}_{k,k}(t) - h^{k;-k}\left (P(t),v(t)\right )\right ) dt+\\
& - Z^{-k;\varphi}_{k,k}(t) dW^k(t)\, ,\\
\varphi(\hat{\tau}^{n}) &= - v^{k;-k}(\hat{\tau}^{n}) \, ,
\end{cases}
\end{equation}
}

where we have introduced the function
\[
h^{k;-k}\left (P(t),v(t)\right ) := v(t)+ \sigma(t)\nu(t)P(t)+Z^P(t)\nu(t)+P(t)b(t)\, .
\]

Notice that, from equation \eqref{EQN:SolPVLastPhi}, we have that $\varphi$ is a BSDE with linear generator, so that its solution is explicitly given by
\[
\varphi^{k;-k}(t) = -\Gamma^{-1}(t) \mathbb{E}_t \left [ \Gamma(\hat{\tau}^{n}) v^{k;-k}(\hat{\tau}^{n}) - \int_t^{\hat{\tau}^{n}} \Gamma(s)h^{k;-k}(P(s),v(s)) ds \right ]\, ,
\]
where $\Gamma$ solves
\[
\begin{split}
d \Gamma(t) &= \Gamma(t)\left [\left (P^{k;-k}(t)-\mu^{k;-k}(t)\right ) dt +\sigma^{k;-k}(t) dW(t)\right ]\, ,\\
\Gamma(0) &= 1\, .
\end{split}
\]

Moreover, by \cite[Th. 5.2, Th. 5.3]{Tan}, it follows that equation \eqref{EQN:SolPVLast} admits a unique adapted solution on $[0,\hat{\tau}^{n}]$. Therefore, iterating the above analysis  for any $k =1,\dots,n$, we gain  the optimal solution to the last control problem.
Having solved  the last control problem, we can consider the second last control problem. Assuming, with no loss of generality, that nodes $(k,l)$ are left, all subsequent computation has to be carried out for any possible couple $k=1,\dots,n$, $l=k+1,\dots,n$.

By Theorem \ref{THM:SMP}, the optimal pair $\left (\bar{X}^{i},\bar{\alpha}^{i}\right )$, $i=k,l$, satisfies, component--wise, the following forward--backward system for $i=k,l$,

{\footnotesize
\begin{equation}\label{EQN:FBSDEL}
\begin{cases}
d X^{i;-(k,l)} (t) &= \left (\mu^{i;-(k,l)}(t) X^{i;-(k,l)}(t) +b^{i;-(k,l)}(t) + \alpha^{i;-(k,l)}(t)\right ) dt +\\
&+\left ( \sigma^{i;-(k,l)}(t) X^{i;-(k,l)}(t)+\nu^{i;-(k,l)}(t)\right )d W^i(t)\, , \\
X^{i;-(k,l)} (\tau^{\pi^{n-2}}) &= X^{i;n-2} (\tau^{\pi^{n-2}})\, ,\\
-dY^{i;-(k,l)} (t) &= \left ( \mu^{i;-(k,l)}(t) Y^{i;-(k,l)}(t) + \sigma^{i;-(k,l)} Z^{i;-(k,l)}_{i,i}(t)\right )dt +\\
& +\left ( \gamma^{i;-(k,l)} X^{i;-(k,l)}(t)- \gamma^{i;-(k,l)} v^{i;-(k,l)}(t)\right ) dt - \sum_{j=k}^l Z^{i;-(k,l)}_{i,j}(t) dW^j(t)\, ,\\
Y^{i;-(k,l)} (\hat{\tau}^{n-1}) &= \gamma^{i;-(k,l)} X^{i;-(k,l)}(\hat{\tau}^{n-1})-\gamma^{i;-(k,l)} v^{i;-(k,l)}(\hat{\tau}^{n-1})+ \bar{Y}^{k;n-1}(\hat{\tau}^{n-1})\, ;
\end{cases}
\end{equation}
}

in what follows we will denote by $Z_j$ the $j-$th $n-$dimensional column of $Z$ in equation \eqref{EQN:Back2Exp}. Note that the only non null entries of $Z$ will be $Z_{i,j}$, for $i,j=k,l$. Also, for the sake of simplicity, we will avoid to use the notation $X^{i;-(k,l)}$, $i=k,l$, only using $X^i$, $i=k,l$, instead.

Mimicking the same method earlier used, we
again guess the solution of the backward component $Y^i$ to be of the form
\begin{equation}\label{EQN:GuessBL}
-Y^i(t) = P^i(t) X^i(t) - \varphi^i(t) \, ,\quad i=k,l\, ,
\end{equation}
for $P^i$ and $\varphi^i$, $i=k,l$, a $\RR-$valued process.

Because of the particular form of equation \eqref{EQN:FBSDEL}, the $i-$th component of the BSDE $Y$ depends only on the $i-$th component of the forward SDE $X$, the matrix $P$ has null entry off the main diagonal, namely it has the form
\[
P(t) =  \left( \begin{array}{ccccc}
 0 & 0 & 0&0&0 \\
 0 & P^{k}(t) & 0 &0&0\\
 0 & 0 & 0 &0&0\\
 0 & 0 & 0 &P^l(t) &0 \\
 0 & 0 & 0 &0&0 \end{array} \right)\, ,
\]
similarly for $\varphi$.

Let us assume that $(P^i(t),Z^{i;P}(t))$, $i=k,l$ solves
\[
\begin{split}
-dP^i(t) &= F^{i;P}(t) dt - \sum_{j=k}^l Z^{P}_j(t) dW^j(t)\, ,\\
P^i(\hat{\tau}^{n-1}) &= \gamma^i  - P^{i}(\hat{\tau}^{n-1}) \Ind{\hat{\tau}^{n-1} = \tau^{-i}}\, ,
\end{split}
\]
and that $(\varphi^i(t),Z^{i;\varphi}(t))$ solves
\[
\begin{split}
-d\varphi^i(t) &= F^{i;\varphi}(t) dt - \sum_{j=k}^l Z_j^\varphi(t) dW^j(t)\, ,\\
\varphi(\hat{\tau}^{n-1}) &= - \gamma^i v^i(\hat{\tau}^{n-1}) + \varphi^i(\hat{\tau}^{n-1})\Ind{\hat{\tau}^{n-1} = \tau^{-i}}\, .
\end{split}
\]

From the first order condition we have  that the optimal control is of the form
\begin{equation}\label{EQN:OPtALbis}
\bar{\alpha}^i = -Y^i(t) = P^i(t) X^i(t) - \varphi^i(t) \, .
\end{equation}
Then, again applying the It\^{o} formula, we have
{\footnotesize
\begin{equation}\label{EQN:Anstaz22bis}
\begin{split}
&\left ( \mu^i(t) \varphi^i(t) - \mu^i(t) P^i(t)X^i(t) + \sigma^i(t) Z_{ii}(t) + \gamma^i X^i (t)- \gamma^i v^i(t)\right ) dt - \sum_{j=k}^l Z_{ij}(t) dW^j(t) =\\
&= -dY^i(t) = d(P^i(t) X^i(t)) - d \varphi^i(t) =\\
&= \left (- F^{i;P} (t) X^i(t) + P^i(t) \mu^i(t) X^i(t) +P^i(t)b^i(t) + P^i(t) \alpha^i(t) + \right )+F^{i;\varphi}(t) dt+\\
&+ \left (\sum_{j=l}^k\left ( Z^{i;P}_j (t) \rho^{ij} \sigma^i(t) X^i(t)+Z^{i;P}_j (t) \rho^{ij} \nu^i(t)\right )\right ) dt\\
&+ \sum_{j=k}^l Z^{i;P}_j (t) X^i(t) dW^j(t) + P^i(t)\left ( \sigma^i(t) X^i (t) + \nu^i(t)\right )dW^i(t) - \sum_{j=1}^2 Z^{i;\varphi}_j(t) dW^j(t) =\\
&=\left ( -F^{i;P} (t) +  P^i(t) \mu^i(t) + \sum_{j=k}^l Z^{i;P}_j (t) \rho^{ij} \sigma^i(t) \right )  X^i(t) dt +  P^i(t) \alpha^i(t) dt +\\
&+F^{i;\varphi}(t) dt+ \sum_{j=l}^kZ^{i;P}_j (t) \rho^{ij} \nu^i(t) dt +P^i(t)b^i(t) dt+\\
&+\left (Z^{i;P}_i (t) +P^i(t) \sigma^i(t)\right ) X^i(t) dW^i(t) + \sum_{\substack{j=k \\ j\not =i}}^l Z^{i;P}_j (t) X^i(t) dW^j(t)+ P^i(t)\nu^i(t) - \sum_{j=k}^l Z^{i;\varphi}_j(t) dW^j(t) \, .
\end{split}
\end{equation}
}

Thus, substituting equation \eqref{EQN:OPtALbis} into equation \eqref{EQN:Anstaz22bis}, and proceeding as for \eqref{EQN:Sys1LM2}, we have
\begin{equation}\label{EQN:Sys1LMA2bis}
F^{i;P} (t) = -\left (P^i(t)\right )^2  + \left (\sigma^i(t)\right )^2 P^i(t) + \sum_{j=k}^l Z^{i;P}_j (t) \ell^{ij} \sigma^i(t) - \gamma^i  \, ,
\end{equation}
with
\[
\ell^{ij} := 
\begin{cases}
\rho^{ij} & i \not = j\, ,\\
2 & i=j\, ,
\end{cases}
\]
together with
{\footnotesize
\[
F^{i;\varphi}(t) =(\mu^{i}-P^i(t)) \varphi(t)+ \sigma^i Z^{i;\varphi}_i(t) - \sum_{j=l}^kZ^{i;P}_j (t) \rho^{ij} \nu^i dt  -P^i(t)\nu^i(t) dt -\gamma^i v^i(t)-\sigma^i(t) \nu^i(t)P^i(t)\, .
\]
}
Turning back, for the sake of clarity, to use the extended notation dropped before, we have that $\bar{\alpha}^{i;-(k,l)}(t)$, $i=k,l$, is given by 
\[
\bar{\alpha}^{i;-(k,l)}(t) = P^{i;-(k,l)}(t) X^{i;-(k,l)}(t) + \varphi^{i;-(k,l)}(t)\, ,
\]
where $P^{i;-(k,l)}$ and $\varphi^{i;-(k,l)}$ are solutions, for $i=k,l$, to the BSDEs 

\begin{equation}\label{EQN:SolPVSecondLast}
\begin{cases}
-dP^{i;-(k,l)}(t) &= \left (-P^{i;-(k,l)}(t)^2  + (\sigma^{i;-(k,l)})^2 P^{i;-(k,l)}(t)\right )dt +\\
&+\left ( \sum_{j=k}^l Z^{-(k,l);P}_j (t) \ell^{ij} \sigma^{i;-(k,l)}(t) - \gamma^{i;-(k,l)}\right )dt - Z^{-(k,l);P}_{i,i}(t) dW^i(t)\, ,\\
P^{i;-(k,l)}(\hat{\tau}^{n-1}) &=\gamma^{i;-(k,l)} - P^{i,-i}(\hat{\tau}^{n-1}) \Ind{\hat{\tau}^{n-1} = \tau^{-i}} \, ,\\
\end{cases}
\end{equation}

\begin{equation}\label{EQN:SolPVSecondLast2}
\begin{cases}
-d\varphi^{i;-(k,l)}(t) &= \left ((\mu^{i;-(k,l)}(t)-P^{i;-(k,l)}(t)) \varphi^{i;-(k,l)}(t)+ \sigma^{i;-(k,l)}(t) Z^{-(k,l);\varphi}_{i,i}(t)\right )dt+\\
& -h^{i;-(k,l)}(P^{i;-(k,l)}(t), v^{i;-(k,l)}(t)) dt - Z^{-(k,l);\varphi}_{i,i}(t) dW^i(t)\, ,\\
\varphi(\hat{\tau}^{n-1}) &= - \gamma^{i;-(k,l)} v^{i;-(k,l)}(\hat{\tau}^{n-1}) + \varphi^{i,-i}(\hat{\tau}^{n-1})\Ind{\hat{\tau}^{n-1} = \tau^{-i}}\, ,
\end{cases}
\end{equation}

with
\[
\begin{split}
h^{i;-(k,l)}(P^{k;-(k,l)}(t),,v ^{k;-(k,l)}(t)) &=  \sum_{j=l}^kZ^{i;P}_j (t) \rho^{ij} \nu^{i;-(k,l)} dt  +P^{i;-(k,l)}(t)\nu^{i;-(k,l)}(t) dt +\\
&+\gamma^{i;-(k,l)} v^{i;-(k,l)}(t)+\sigma^i(t) \nu^{i;-(k,l)}(t)P^{i;-(k,l)}(t)\, .
\end{split}
\]
Let us underline that equations \eqref{EQN:SolPVSecondLast}--\eqref{EQN:SolPVSecondLast2} have to be solved for any possible couple $k=1,\dots,n$, $l=k+1,\dots,n$.
As before, by the linearity of the generator of $\varphi^i$ in equation \eqref{EQN:SolPVSecondLast2}, we have
\[
\begin{split}
\varphi^{i;-(k,l)}(t) &= -\left (\Gamma^i(t)\right )^{-1} \mathbb{E}_t \left [ \Gamma^i(\hat{\tau}^{n-1}) \left ( \varphi^{i,-i}(\hat{\tau}^{n-1})\Ind{\hat{\tau}^{n-1} = \tau^{-i}}\right )- \gamma^i v^i(\hat{\tau}^{n-1})\right ] + \\
&-\left (\Gamma^i(t)\right )^{-1} \mathbb{E}_t \left [ \int_t^{\hat{\tau}^{n-1}} \Gamma^i(s) h^{i;-(k,l)}(P^{i;-(k,l)}(s),v^{i;-(k,l)}(s)) ds \right ]\, ,
\end{split}
\]
where $\Gamma^i$ is the solution to
\[
d \Gamma^i(t) = \Gamma^i(t)\left [\mu^i(t) dt +\sigma^i(t) dW^i(t)\right ]\, ,\quad \Gamma^i(0) = 1\, .
\]
Hence, equation \eqref{EQN:SolPVSecondLast} admits a unique adapted solution on $[0,\hat{\tau}^{n-1}]$, see \cite[Th.5.2, Th. 5.3]{Tan}.

Analogously, via a backward induction, we can solve the first control problem, that is we solve, for $i = 1,\dots,n$,

{\footnotesize
\begin{equation}\label{EQN:FBSDEL0}
\begin{cases}
d X^{i;0} (t) = \left (\mu^{i;0}(t) X^{i;0}(t) + b^{i;0}(t)+ \alpha^{i;0}(t)\right ) dt + \left (\sigma^{i;0}(t) X^{i;0}(t) + \nu^{i;0}(t)\right )d W^i(t)\, , \\
X^{i;0} (0) = x^i_0\, ,\\
-dY^{i;0} (t) = \left ( \mu^{i;0}(t) Y^{i;0}(t) + \sigma^{i;0} Z_{i,i}^0(t) + \gamma^{i;0} X^{i;0}(t)- \gamma^{i;0} v^{i;0}(t)\right ) dt - \sum_{j=1}^n Z_{i,j}^0(t) dW^j(t)\, ,\\
Y^{i;0} (\hat{\tau}^{1}) = \gamma^{i;0} X^{i;0}(\hat{\tau}^{1})-\gamma^{i;0} v^{i;0}(\hat{\tau}^{1})+ Y^{i;1}(\hat{\tau}^{1})\Ind{\hat{\tau}^{1} \not= \tau^{i}}\, ,
\end{cases}
\end{equation}
}

resulting, exactly repeating what considered so far,
to consider an optimal control of the form
\[
\alpha^{i;0}(t) =  -Y^{i;0}(t) = P^{i;0}(t) X^{i;0}(t) - \varphi^{i;0}(t) \, ,
\]
and
\begin{equation}\label{EQN:SolPVLast0}
\begin{cases}
-dP^{i;0)}(t) &= \left (-(P^{i;0}(t))^2  + (\sigma^{i;0}(t))^2 P^{i;0}(t) + \sum_{j=1}^n Z^{0;P}_j (t) \ell^{ij} \sigma^{i;0}(t) - \gamma^{i;0}\right )dt +\\
&- Z^{0;P}_{i,i}(t) dW^i(t)\, ,\\
P^{i;0}(\hat{\tau}^{1}) &=\gamma^{i;0}- P^{i;1}(\hat{\tau}^{1})\Ind{\hat{\tau}^{1} \not= \tau^i} \, ,\\
\end{cases}
\end{equation}

\begin{equation}\label{EQN:SolPVLast02}
\begin{cases}
-d\varphi^{i;0}(t) &= \left ((\mu^{i;0)}-P^{i;0}(t)) \varphi^{i;0}(t)+ \sigma^{i;0}(t) Z^{0;\varphi}_{i,i}(t) - \gamma^{i;0} v^{i;0}(t)\right ) dt+\\
& - Z^{0;\varphi}_{i,i}(t) dW^i(t)\, ,\\
\varphi^{i;0}(\hat{\tau}^{1}) &= \varphi^{i;1}(\hat{\tau}^{1})\Ind{\hat{\tau}^{1} \not= \tau^i} - \gamma^{i;0}v^{i;0}(\hat{\tau}^{1})\, ,
\end{cases}
\end{equation}
with
\[
\begin{split}
h^{i;0}(P^{i;0}(t),v^{i;0}(t)) &=  \sum_{j=1}^n Z^{i;P}_j (t) \rho^{ij} \nu^{i;0} dt  +P^{i;0}(t)\nu^{i;0}(t) dt +\\
&+\gamma^{i;0} v^{i;0}(t)+\sigma^{i;0}(t) \nu^{i;0}(t)P^{i;0}(t)\, ,
\end{split}
\]
which concludes the proof.
\end{proof}

\section*{Acknowledgement}
The authors wish to thank Prof. Luciano Campi for his stimulating comments and enlightening suggestions. The authors also like to thank the group \textit{Gruppo Nazionale per l'Analisi Matematica, la Probabilità e le loro Applicazioni} (GNAMPA) for the financial help which has supported the present research within the project \textit{Stochastic Partial Differential Equations and Stochastic Optimal Transport with Applications to Mathematical Finance}.

\cleardoublepage

\end{document}